\documentclass[a4paper,11pt]{article}

\usepackage[top=3.5cm, bottom=3.5cm, left=2.5cm, right=2.5cm]{geometry} 

\usepackage{fancyhdr}
\usepackage{titlesec}
\titleformat{\chapter}[hang]
  {\vspace{-1cm}\normalfont\huge\bfseries\Huge}
  {\chaptertitlename\ \thechapter}{10pt}{\,\,\,\,\sc\Huge}
\usepackage{titletoc}

\usepackage[sc]{caption}
\usepackage{arydshln}
\usepackage{graphicx}
\usepackage{subfigure}
\usepackage{here}
\usepackage{float}
\usepackage{amsthm}
\usepackage{amssymb}
\usepackage{epstopdf}
\usepackage[latin1]{inputenc} 
\usepackage[english]{babel} 
\usepackage{amsmath}
\usepackage{amssymb}
\usepackage{mathrsfs}
\usepackage{enumerate}
\usepackage{color}
\usepackage[colorlinks=false,linkcolor=blue]{hyperref}
\usepackage[nottoc, notlof, notlot]{tocbibind}
\usepackage{titling}
\usepackage{mathtools}
\usepackage[]{pdfpages}
\usepackage[numbers,sort&compress]{natbib}
\newcommand{\subtitle}[1]{%
  \posttitle{%
    \par\end{center}
    \begin{center}\large#1\end{center}
    \vskip0.5em}%
}
\usepackage{pgf,tikz}
\usepackage{mathrsfs}
\usetikzlibrary{arrows}
\definecolor{uuuuuu}{rgb}{0.26666666666666666,0.26666666666666666,0.26666666666666666}
\usepackage{authblk}

\DeclareFontFamily{U}{matha}{\hyphenchar\font45}
\DeclareFontShape{U}{matha}{m}{n}{
      <5> <6> <7> <8> <9> <10> gen * matha
      <10.95> matha10 <12> <14.4> <17.28> <20.74> <24.88> matha12
      }{}
\DeclareSymbolFont{matha}{U}{matha}{m}{n}
\DeclareFontSubstitution{U}{matha}{m}{n}

\DeclareFontFamily{U}{mathx}{\hyphenchar\font45}
\DeclareFontShape{U}{mathx}{m}{n}{
      <5> <6> <7> <8> <9> <10>
      <10.95> <12> <14.4> <17.28> <20.74> <24.88>
      mathx10
      }{}
\DeclareSymbolFont{mathx}{U}{mathx}{m}{n}
\DeclareFontSubstitution{U}{mathx}{m}{n}

\DeclareMathDelimiter{\vvvert}{0}{matha}{"7E}{mathx}{"17}

\newcommand{\e}{\mathrm{e}}
\newcommand{\PP}{\mathbb{P}}

\newcommand{\E}{\mathbb{E}}

\newcommand{\R}{\mathbb{R}}

\newcommand{\N}{\mathbb{N}}

\newcommand{\Sph}{\mathbb{S}}

\newcommand{\dd}{\mathrm{d}}

\newcommand{\OO}{\mathcal{O}}

\newcommand{\1}{\mathbf{1}}
\newcommand{\MM}{\mathscr{M}}

\newcommand{\T}{\mathbb{T}}

\DeclareMathOperator{\supess}{supess}

\DeclareMathOperator{\Supp}{Supp}

\DeclareMathOperator{\Tr}{Tr}

\addto\captionsenglish{%
}

\newtheoremstyle{exemple}
	{\topsep}
	{\topsep}
	{\normalfont}
	{}
	{\sc}
	{.}
	{5pt plus 1pt minus 1pt}
	{}

\newtheorem{theorem}{Theorem}

\newtheorem{proposition}{Proposition}[section]

\newtheorem{lemma}{Lemma}[section]

\newtheorem{corollary}{Corollary}[section]

\theoremstyle{definition}
\newtheorem{definition}{Definition}[section]

\theoremstyle{example}

\newtheorem{rem}{Remark}[section]

\newtheorem{assumption}[]{Assumption}

\title{\bf Propagation of chaos and moderate interaction for a piecewise deterministic system of geometrically enriched particles}
\date{}
\author{Antoine \sc{Diez}}
\affil{
Department of Mathematics, Imperial College London, South
Kensington Campus,
London, SW7 2AZ, UK,

\url{antoine.diez18@imperial.ac.uk}

}
\begin{document}
\renewcommand{\chaptername}{}
\renewcommand{\proofname}{\sc{Proof}}
\maketitle
\abstract{In this article we study a system of $N$ particles, each of them being defined by the couple of a position (in $\R^d$) and a so-called orientation which is an element of a compact Riemannian manifold. This orientation can be seen as a generalisation of the velocity in Vicsek-type models such as \cite{degond2008continuum,degondfrouvellemerino17}. We will assume that the orientation of each particle follows a jump process whereas its position evolves deterministically between two jumps. The law of the jump depends on the position of the particle and the orientations of its neighbours. In the limit $N\to+\infty$, we first prove a propagation of chaos result which can be seen as an adaptation of the classical result on McKean-Vlasov systems \cite{sznitman89} to Piecewise Deterministic Markov Processes (PDMP). As in \cite{jourdainmeleard}, we then prove that under a proper rescaling with respect to $N$ of the interaction radius between the agents (moderate interaction), the law of the limiting mean-field system satisfies a BGK equation with localised interactions which has been studied as a model of collective behaviour in \cite{toto}. Finally, in the spatially homogeneous case, we give an alternative approach based on martingale arguments.}\\




\section{Introduction}

In many systems of interacting agents, a wide range of self-organised collective behaviours can be observed. It includes the flock of birds \cite{hildenbrandt2010self}, sperm motion \cite{degond2015multi}, opinion dynamics \cite{degond2017continuum} etc. These systems are composed of many interacting agents and the interacting mechanisms which lead to the observed phenomena are not always known or tractable. To study such complex systems a common procedure consists in writing kinetic or fluids models, i.e. a model for the evolution of an observable macroscopic quantity, for instance the spatial density of agents or their average velocity. Starting from an Individual Based Model (IBM) which describes the interacting mechanisms at the level of the agents (microscopic description), the relevance of a macroscopic model can be justified when the number of agents tends to infinity. The present work aims to be a contribution to the rigorous derivation of kinetic models of collective behaviour such as \cite{toto} from an IBM. 

\subsection{Models for collective dynamics: prototype and extensions}\label{prototypical}

Complex collective behaviours at the macroscopic scale can be observed even from IBM with simple interacting rules between the agents. An example of such IBM is given by the classical Vicsek model \cite{vicsek1995novel}. In this model $N$ oriented agents evolve in $\R^d$. Their velocity is assumed to have constant modulus and then is only prescribed by their orientation which is a unit vector in the sphere $\Sph^{d-1}$. At discrete times, each agent computes the average orientation of its neighbours and takes as a new orientation a random perturbation of this average. A continuous time version of this model can be described by a Piecewise Deterministic Markov Process (PDMP) in the phase space $\R^d\times\Sph^{d-1}$ (see \cite{azais2014piecewise,davis84} for a review on PDMP). Each of the $N$ agents is defined at time $t\geq0$ by its position $X^i_t\in\R^d$ and its orientation $\omega^i_t\in\Sph^{d-1}$, $i\in\{1,\ldots,N\}$ with the following rules: 
\begin{enumerate}[$\bullet$]
\item To each agent is attached an internal clock (independent of the other agents' clocks), that is a sequence of jump times $T^i_1,T^i_2,\ldots$ such that the increments $T^i_{n+1}-T^i_n$ are independent and identically distributed with respect to an exponential law with constant parameter. 
\item Between two jump times, the agent follows a deterministic flow: 
\[\forall n\geq1,\,\,\forall t\in[T^i_n,T^i_{n+1}),\,\,\,\left\{
\begin{array}{rcl}
X^i_t &=& X^i_{T_n}+(t-T^i_n)\omega^i_t \\
\omega^i_t &=& \omega^i_{T_n}
\end{array}
\right. , \]
i.e. the agent moves at constant speed $\omega^i_{T_n}$. 
\item At each jump time $T^i_n$, a jump occurs (for the orientation only), given by a transition kernel: 
\[Q_{T^i_n}\left(X^i_{T^i_n},\omega_{T^i_n}\right) = \delta_{X^i_{T^i_n}}\otimes M\big[X^i_{T^i_n}\big]\]
where $M[X^i_{T^i_n}]$ is a probability distribution function (PDF) on $\Sph^{d-1}$ which depends on the position $X^i_{T_n}$. A typical choice for this PDF would be a von Mises type distribution on $\Sph^{d-1}$, the mean of which is in the direction of the average orientation of the other agents in a neighbourhood of $X^i_{T^i_n}$ : 
\begin{equation}\label{vonmises}M\big[X^i_{T^i_n}\big](\omega) = \frac{\e^{J\big(X^i_{T^i_n}\big)\cdot\omega}}{\mathcal{Z}}\end{equation}
where $\mathcal{Z}$ is a normalisation constant and 
\begin{equation}\label{Jvicsek}J\Big(X^i_{T^i_n}\Big) = \sum_{j=1}^N K\left(X^i_{T^i_n}-X^j_{T^i_n}\right)\omega^j_{T^i_n}\in \R^d\end{equation}
is a local non-normalised average orientation (also called a flux) and $K$ is typically the indicator of a ball centered at the origin with integral 1 (observation kernel). The von Mises distribution \eqref{vonmises} is a unimodal distribution with a maximum at the point $J(X^i_{T^i_n})/|J(X^i_{T^i_n})|\in\Sph^{d-1}$ when $J\big(X^i_{T^i_n}\big)\ne0$ and is the uniform distribution on the sphere when $J\big(X^i_{T^i_n}\big)=0$. 
\end{enumerate}

As the number of agents $N$ tends to infinity, the kinetic model associated to the previous IBM can be formally derived and is given by the following Partial Differential Equation (PDE) on the density of agents $f=f(t,x,\omega)$ on $\R^d\times\Sph^{d-1}$ : 
\begin{equation}\label{bgkvicsek}\partial_t f + \omega\cdot\nabla_x f = \rho_f M_{J_{K*f}}-f\end{equation}
where
\[\rho_f(x) = \int_{\Sph^{d-1}} f(t,x,\omega)\dd\omega\]
is the local spatial density of agents,
\begin{equation}\label{Jbgk}J_{K*f}=J_{K*f}(x) = \iint_{\R^d\times\Sph^{d-1}} K(x-y)\omega f(t,y,\omega)\dd y \dd\omega \in\R^d\end{equation}
is the local non-normalised average orientation (or flux) of the agents and
\[M_{J_{K*f}}= M_{J_{K*f}(x)}(\omega) = \frac{\e^{J_{K*f}(x)\cdot\omega}}{\mathcal{Z}}\]
is the von Mises distribution on $\Sph^{d-1}$ with parameter $J_{K*f}(x)$. This type of equations is referred in the literature as a Bhatnagar-Gross-Krook (BGK) type equation \cite{bhatnagar1954model}.\\ 

This prototypical model (the IBM and its kinetic version) will be the starting point of this article. A very close form of this model and a formal derivation of the kinetic model from the IBM have first been written in \cite{dimarcomotsch16}. The main difference with our model is the computation of the average orientation: in \cite{dimarcomotsch16}, the local fluxes \eqref{Jvicsek} and \eqref{Jbgk} are respectively replaced by a normalised version:
\[\Omega^i_{T^i_n} = \frac{J\big(X^i_{T^i_n}\big)}{|J\big(X^i_{T^i_n}\big)|}\in\Sph^{d-1}\,\,\,\text{and}\,\,\,\Omega_f =\frac{J_{K*f}}{|J_{K*f}|}\in\Sph^{d-1}\]
and can therefore be interpreted as normalised average orientations in $\Sph^{d-1}$. From a modelling point of view, in our model, when a jump occurs, the bigger the norm of the flux, the more peaked the von Mises distribution is and therefore the new orientation of the agent is (expectedly) closer to the local normalised average orientation. From a mathematical point of view the normalised version of the model studied in \cite{dimarcomotsch16} is more singular, as the normalised average orientation is not defined when the flux \eqref{Jvicsek} or \eqref{Jbgk} is equal to zero. In this case, in our model, the agent will simply draw a new orientation uniformly on the sphere. It has also been shown in \cite{degondfrouvelleliu15} that complex behaviours and in particular phase transitions only appear in models with a non-normalised flux. Without this singularity problem, the aim of this article is to prove the well-posedness of the PDE \eqref{bgkvicsek} and to rigorously justify its derivation from the IBM by taking the limit (in a certain sense) when $N$ tends to infinity. \\ 

Starting from the original work of Vicsek \cite{vicsek1995novel} and apart from the previously cited model \cite{dimarcomotsch16}, different models of collective dynamics have been proposed in the literature. Continuous time models include \cite{degond2008continuum} where the analog of our IBM is given by a system of Stochastic Differential Equations (SDEs) and the kinetic version analogous to \eqref{bgkvicsek} is a given by a Fokker-Planck equation. The links between the IBM and the kinetic Fokker-Planck equation have been studied in \cite{bolley2011stochastic} and the analysis of the kinetic model has been carried out in \cite{degondfrouvelleliu15}. In addition to the choice of form of the IBM (PDMP or SDE), a new perspective has been investigated in \cite{degondfrouvellemerino17,degondfrouvellemerinotrescases18} to model agents with a more complex geometrical structure. In these models, the velocity of the agents is replaced by their so called body-orientation (in dimension 3) which is modelled either by a rotation matrix or by a quaternion. A PDMP model and in particular its kinetic version have been studied (in a space homogeneous case) in \cite{toto}.\\

 In this article we will focus on IBMs based on a PDMP. We will rigorously derive their kinetic versions, which will be given by BGK type equations such as \eqref{bgkvicsek}. To encompass the  ``geometrically enriched'' models \cite{toto} and the model initially described, we will write a model in an abstract framework where the orientation of the agents is an element of a compact Riemannian manifold $\mathscr{M}$ : the choice $\mathscr{M}=\Sph^{d-1}$ gives the model \eqref{bgkvicsek} and the choice $\mathscr{M}=SO_3(\R)$ gives the model studied in \cite{toto}. See \cite{toto,degond2018alignment} and the references therein for a review and a comparison of these models. \\

\subsection{Challenges, methods and contributions}

From a mathematical point of view, the classical mathematical kinetic theory of gases provides well adapted tools and concepts to study multi-scale models and limiting procedures between models at different scales. Its infancy goes back to Boltzmann and Hilbert at the end of the XIXth century and the beginning of the XXth century.  The subject has been of growing interest throughout the XXth century and up to now. A review of the main concepts can be found in \cite{degond2004macroscopic,cercignani2013mathematical}. Following these ideas, in \cite{degond2008continuum}, the authors extended the classical Vicsek model \cite{vicsek1995novel} to a time-continuous IBM from which kinetic and fluid models can be derived. The analysis of the kinetic formulation of the Vicsek model exhibits collective behaviour such as phase transitions \cite{degondfrouvelleliu15} between ordered and disordered dynamics which can be observed in biological systems such as the flock of birds. \\

However, the rigorous derivation of kinetic PDEs from IBMs remains one of the main challenges in classical kinetic theory. The type of interactions which is studied in the present work enters into the class of mean-field interactions with randomness. The mean-field theory approach consists in approximating the trajectory of a typical particle by considering that it evolves in a force field constructed by averaging the interactions between the neighbouring particles. The rigorous derivation of the associated kinetic PDE models such as the Vlasov equation have been initiated in particular by Dobrushin \cite{dobrushin1979vlasov}, Braun and Hepp \cite{braun1977vlasov} in a deterministic framework or by McKean \cite{mckean1966class,mckean1967propagation} and Kac \cite{kac1956foundations} to incorporate randomness. A key notion in this context is the so-called \textit{molecular chaos} or more precisely the \textit{propagation of chaos} property formalised by Kac \cite{kac1956foundations} in order to investigate how the statistical independence between the particles evolves in time depending on the interaction rules. Reviews of old and recent results on systems of particles with mean-field interactions can be found for instance in \cite{jabin2017mean,carrillo2014derivation,golse2016dynamics}. \\

The classical work by Sznitman \cite{sznitman89} provides an efficient method based on coupling arguments to prove the propagation of chaos property for McKean-Vlasov systems of Stochastic Differential Equations (SDEs). In particular, it has been successfully applied for the continuous version of the Vicsek model in \cite{bolley2012mean}. Coupling arguments in the context of PDMPs have been used in \cite{monmarche} to prove well-posedness and trend to equilibrium of systems of PDMP with mean-field interaction. In the present work, we extend the method of coupling of \cite{sznitman89} to the piecewise deterministic setting and we prove a propagation of chaos property for a geometrically enriched system of PDMPs. This property is twofold: firstly it states that if the agents are initially chosen independently, then this independence is propagated at any finite time as the number of agents $N$ increases. Secondly it states that the law of any agent (they are identically distributed by symmetry) converges in a certain sense to the solution of a BGK equation (namely \eqref{bgkvicsek} in the case $\MM=\Sph^{d-1}$). An explicit convergence rate with respect to $N$ is obtained.  \\ 

In addition to this propagation of chaos property and following the approach of \cite{jourdainmeleard} we will also prove a moderate interaction property. It states that under an appropriate rescaling of the size of the neighbourhood of the agents with respect to the total number of agents, the interaction can be made purely local, which means that we can take $K=\delta_0$ in the BGK equation. This rescaling is achieved by taking in the IBM an observation kernel of the form: 
\[K^N(x)=\frac{1}{\varepsilon_N^{d}}K_0\left(\frac{x}{\varepsilon_N}\right)\]
where $K_0$ is a fixed observation kernel and $\varepsilon_N\to0$ slowly enough. For our prototypical model described in subsection \ref{prototypical}, it proves that the law of any agent converges (in a certain sense) as $N\to+\infty$ to the solution of the following BGK equation \eqref{bgkvicsek}:
\[\partial_t f +\omega\cdot\nabla_x f = \rho_f M_{J_f}-f\]
where
\[\rho_f(x) = \int_{\Sph^{d-1}}f(t,x,\omega)\dd\omega\,\,\,\text{and}\,\,\,J_f(x) = \int_{\Sph^{d-1}}\omega f(t,x,\omega)\dd\omega\]
and $M_{J_f}$ is the von Mises distribution on $\Sph^{d-1}$ with parameter $J_f$. This result will follow from the explicit bound in $N$ obtained in the proof of the propagation of chaos and from classical compactness arguments to pass to the limit inside the BGK equation ``with kernel interaction'' when $K^N\to\delta_0$ (see Equation \eqref{bgkvicsek}). This will require to prove \emph{ad hoc} regularity properties on the solution of the BGK equation in the geometrically enriched specific setting considered (in particular the equicontinuity and stability under translations of the sequence of solutions associated to the sequence of kernels). \\

 The terminology ``moderate interaction'' comes from \cite{Oelschl_ger_1985} where propagation of chaos and moderate interaction properties are proved using a different approach, based on martingale arguments. As explained in \cite{jourdainmeleard}, in the model studied by Oelschläger, the random part is given by a constant diffusion matrix which does not depend on the current state of the system (i.e. of the empirical measure of the agents). This is not the case in most models of collective dynamics and in particular in the model studied in the present article. However, in a space homogeneous setting, we give an alternative proof of the propagation of chaos property based on martingale arguments. Let us also mention that martingales techniques have also recently been used in \cite{friesen2018stochastic} to prove propagation of chaos for a different model of collective dynamics known in the literature as the stochastic Cucker-Smale model.\\

The organisation of the article is as follows. In the first section we describe the general geometrical framework which encompasses our prototypical model and its extensions (subsection \ref{sectionencompass}). In this framework, we will study the IBM defined in subsection \ref{sectionIBM} and its kinetic version (subsection \ref{sectionkinetic}). The main results are stated in subsection \ref{sectionmainresults}, in particular Theorem \ref{propagationofchaostheorem} (propagation of chaos) and Theorem \ref{moderateinteractiontheorem} (moderate interaction). Section \ref{propagationofchaossection} is devoted to the proof of the former. The latter is proved in Section \ref{moderateinteractionsection} together with regularity results for the solution of the BGK equation. Finally, in Section \ref{alternativesection}, we give an alternative approach based on martingale techniques in a space-homogeneous model. \\

\subsection{Notations and definitions}
For the convenience of the reader, we collect here the main notations and technical definitions that will be used in the rest of the article. 
\begin{enumerate}[$\bullet$]
\item $\mathcal{P}(E)$ is the space of probability measures on the Polish space $E$. 
\item $\widetilde{\mathcal{P}}(E)$ is the space of probability measures on the measure space $(E,\mu)$ which are absolutely continuous with respect to the measure $\mu$ on $E$. An element of $\widetilde{\mathcal{P}}(E)$ is identified with its associated probability density function.
\item $s^\sharp\mu$ denotes the push-forward measure of the measure $\mu$ on a Polish space $E$ by the measurable map $s:E\to E$. It is defined by $s^\sharp\mu(A) = \mu(s^{-1}(A))$ for any Borel subset $A\subset E$. 
\item $\mathcal{M}_+(E)$ is the space of positive measures on the Polish space $E$.
\item $\mathcal{M}^N(E)$ is the space of empirical measures of size $N\in\N$ on the Polish space $E$, that is to say:
\[\mathcal{M}^N(E):=\left\{\frac{1}{N}\sum_{i=1}^N \delta_{x_i},\,\,(x_1,\ldots,x_N)\in E^N\right\}\subset\mathcal{P}(E)\]
\item $C(E_1,E_2)$ is the space of continuous functions from the (metric) space $E_1$ to the (metric) space $E_2$. When $E_2=\R$ we write $C(E_1)\equiv C(E_1,\R)$. 
\item $C_b(E)$ is the space of bounded real-valued continuous functions on the metric space $(E,d)$. 
\item $L^p(E)$ for $p\in[1,\infty]$ is the space of real-valued measurable functions $f$ on the measure space $(E,\mu)$ such that 
\[\|f\|_{L^p(E)}:=\left(\int_E |f(x)|^p\dd\mu(x)\right)^{1/p}<\infty\,\,\,\text{for}\,\,\,p\in[1,\infty),\]
\[\|f\|_{L^\infty(E)} := \underset{x\in E}{\supess} |f(x)|<\infty\,\,\,\text{for}\,\,\,p=\infty.\]
We write for short $\|\cdot\|_{L^p}\equiv\|\cdot\|_{L^p(E)}$ when the space $E$ is clearly identified. 
\item $\|\varphi\|_{\mathrm{Lip}}$ denotes the Lipschitz norm of the real-valued function $\varphi$ on the metric space $(E,d)$ : 
\[\|\varphi\|_{\mathrm{Lip}}:=\sup_{x\ne y}\frac{|\varphi(x)-\varphi(y)|}{d(x,y)}.\]
\item $X\sim\mu$ for a random variable $X$ and a probability measure $\mu$ means that the law of $X$ is $\mu$
\item $W^1(\mu_1,\mu_2)$ is the Wasserstein-1 distance between two probability measures on the Polish space $(E,d)$ defined by: 
\[W^1(\mu_1,\mu_2):=\inf_{\pi\in\Pi(\mu_1,\mu_2)} \int_{E\times E} d(x,y)\dd\pi(x,y)\]
where $\Pi(\mu_1,\mu_2)$ is the set of all couplings of $\mu_1$ and $\mu_2$ that is to say the set of probability measures $\pi$ on the product space $E\times E$ such that the first marginal of $\pi$ is equal to $\mu_1$ and the second marginal of $\pi$ is equal to $\mu_2$. Equivalently $W^1(\mu_1,\mu_2)$ is defined by 
\[W^1(\mu_1,\mu_2)=\inf_{\substack{X\sim \mu_1\\ Y\sim \mu_2}} \E[d(X,Y)]\]
or (Kantorovich dual formulation):
\begin{equation}\label{kantorovich}W^1(\mu_1,\mu_2)=\sup_{\|\varphi\|_{\mathrm{Lip}}\leq1}\left\{ \int_{E}\varphi(x)d\mu_1(x)-\int_E \varphi(x)d\mu_2(x)\right\}.\end{equation}
\item $\|\mu_1-\mu_2\|_{TV}$ denotes the total variation norm between the probability measures $\mu_1$ and $\mu_2$ on the Polish space $(E,d)$ defined by: 
\[\|\mu_1-\mu_2\|_{TV}:=2\inf_{\substack{X\sim \mu_1 \\ Y\sim\mu_2}}\PP(X\ne Y)\]
or by the duality formula:
\[\|\mu_1-\mu_2\|_{TV}=\sup_{\|\varphi\|_{L^\infty}\leq 1}\left\{ \int_{E}\varphi(x)d\mu_1(x)-\int_E \varphi(x)d\mu_2(x)\right\}.\]
Since $E$ is a Polish space, the supremum can be taken over real-valued bounded continuous functions with $L^\infty$ norm bounded by 1 (equivalence between the total variation distance and the Radon distance on Polish spaces). When $\mu_1$ and $\mu_2$ have a density, their total variation distance is therefore the $L^1$ norm of the difference of the two associated probability density functions. More details about the Wasserstein and Total Variation distances and the proof of these properties can be found in \cite{villanioptimaltransport}.

\end{enumerate}

\section{Abstract framework and main results}\label{generalsetting}
\subsection{Functional set up}\label{sectionencompass}

\subsubsection{Assumptions and definitions in the abstract framework}
From now on we will use the following functional set up.
\begin{enumerate}
\item Let $(\mathscr{M},g)$ be a compact finite dimensional Riemannian manifold of dimension $p$. The Riemannian distance induced by $g$ is denoted by $d$. On the product space $\R^d\times \MM$ we take the metric 
\[\tilde{d}\big((x_1,m_1),(x_2,m_2)\big):=|x_1-x_2|+d(m_1,m_2)\]
and we write for short $\tilde{d}\equiv d$ when there is no possible confusion. The volume form associated to $g$ is assumed to be normalised and will be denoted by $\dd m$ (i.e. $\int_\mathscr{M}\dd m=~1$). We will also assume that $\mathscr{M}$ is isometrically embedded in a euclidean space $E$ where the inner product is denoted by $\cdot$ and the norm by $|\cdot|$. This embedding is never a loss of generality thanks to Nash's embedding theorem \cite{nash1956imbedding}. We will also assume without loss of generality that for all $m\in\MM$, $|m|\leq1$. 
 \item Let $\Phi : \mathscr{M}\to \R^d$ be a \textit{velocity map} which is a continuous and $\lambda$-Lipschitz map for a given constant $\lambda\geq0$.
\item Let $K$ be a smooth observation kernel on $\R^d$, that is to say a radial smooth bounded Lipschitz function which tends to zero at infinity (typically a smoothened version of the indicator of a ball centred at the origin) such that $K\geq 0$ and $\int_{\R^d} K(x)\dd x=1$.
 \item Let $M_J$ be an \textit{interaction law}, defined for any $J\in E$ as a probability density function on $\mathscr{M}$ which satisfies the following assumptions. 
 
\begin{assumption}[Locally bounded]\label{asslocallybounded}
 There exists a function $\alpha=\alpha(a)$ such that for any $a>0$ and any $J\in E$ with $|J|\leq a$, it holds that
\begin{equation}\label{assumptionlocallybounded}
\| M_J\|_{L^\infty(\mathscr{M})}\leq \alpha(a).
\end{equation}
\end{assumption}

\begin{assumption}[Locally Lipschitz]\label{assumptionlocallipschitz}
There exists a function $L=L(a)$ such that for any $a>0$ and any $J\in E$ with $|J|\leq a$, it holds that
\begin{equation}\label{locallylipschitz}
|M_J(m_1)-M_J(m_2)|\leq L(a)d(m_1,m_2).
\end{equation}
\end{assumption}

\begin{assumption}[Flux Lipschitz]\label{assumptionfluxlipschitz}
There exists a function $\theta=\theta(a)$ such that for any $a>0$ and any $J,J'\in E$ with $|J|,|J'|\leq a$, it holds that
\begin{equation}\label{fluxlipschitz}
\|M_{J}-M_{J'}\|_{L^\infty(\mathscr{M})}\leq \theta(a)|J-J'|.
\end{equation}
\end{assumption}

\end{enumerate}

\begin{rem}
The Lipschitz regularity assumptions (for the interaction kernel and for the interaction law) are classical for mean-field type results (see \cite{sznitman89}). The present work essentially focuses on the extension of classical results and techniques to the PDMP framework with geometrical constraints. Recent results which investigate the case of less regular interactions but in the classical McKean-Vlasov framework can be found for instance in \cite{jabin2017mean}. 
\end{rem}

In order to define interaction rules between the agents we now define two objects: the \textit{flux} of a measure which will be a way of constructing an average orientation from a distribution of orientations and the \textit{observation measure}, the purpose of which will be to define the local average orientation around a point in the physical space $\R^d$. 

\begin{itemize}

\item The \textit{flux} of a positive measure $\mu\in\mathcal{M}_+(\mathscr{M})$ on $\mathscr{M}$ is defined by 
\begin{equation}\label{flux}
J_\mu := \int_{\mathscr{M}} m\,\dd\mu(m) \in E.\end{equation}
The interaction law relative to a positive measure $\mu\in\mathcal{M}_+(\mathscr{M})$ is defined as the probability density function $M_{J_\mu}$ on $\MM$ and will be denoted only by $M_\mu$ in the following. \\



\item Given a probability measure $p$ on $\R^d\times\MM$, we define the \textit{observation measure} by taking the convolution product of $p$ with the observation kernel $K$: its purpose is to define the local orientation with respect to $p$ around a point $x\in\R^d$. In particular $p$ can be either the empirical measure of a system of processes given by an IBM or the solution of a BGK equation in a kinetic model.

\begin{definition}[Observation measure]\label{observationmeasure}
Let $p\in\mathcal{P}(\R^d\times \mathscr{M})$ and $x\in\R^d$. The observation measure $K*p(x)\in\mathcal{M}_+(\MM)$ is defined by 
\[\forall \varphi\in C(\MM),\,\,\,\langle \varphi,K*p(x)\rangle:=\iint_{\R^d\times\MM}K(x-y)\varphi(m) \dd p(y,m).\]
\end{definition}

Note that $K*p : x\in \R^d\mapsto K*p(x)\in\mathcal{M}_+(\MM)$ defines a smooth map for the total variation topology on $\mathcal{M}_+(\MM)$. The flux and interaction law will be denoted in this case: 
\begin{equation}\label{notationkernelflux}J_{K*p(x)}\equiv J_{K*p}(x)\,\,\,\,\text{and}\,\,\,\,M_{K*p(x)}(m)\equiv M_{K*p}[x](m).\end{equation}

If $p$ is a probability density function in $L^\infty(\R^d\times\MM)$, then the previous definition makes sense in the degenerate case $K=\delta_0$, and we define
\[\forall \varphi\in C(\MM),\,\,\,\langle \varphi,p(x,\cdot)\rangle:=\int_{\MM}\varphi(m) p(x,m)\dd m.\]
The flux and interaction law will be denoted in this case: 
\[J_{p(x,\cdot)}\equiv J_{p}(x)\,\,\,\,\text{and}\,\,\,\,M_{p(x,\cdot)}(m)\equiv M_{p}[x](m).\]
\end{itemize}

To conclude this section, we point out that the assumptions on the interaction law can be reinterpreted as regularity bounds in the space of probability measures when the interaction law comes from a measure on $\MM$. In particular, since 
\[W^1(M_\mu,M_\nu)\leq \|M_\mu-M_\nu\|_{TV}=\|M_\mu-M_\nu\|_{L^1(\MM)}\leq\|M_\mu-M_\nu\|_{L^\infty(\MM)},\]
the flux-Lipschitz bound \eqref{fluxlipschitz} implies that for $p_1,p_2\in \widetilde{\mathcal{P}}(\R^d\times\MM)$ and $x,x'\in\R^d$ :
\begin{equation}\label{wassersteinboundM}W^1(M_{K*p_1}[x],M_{K*p_2}[x'])\leq C\theta\big(\|K\|_{L^\infty}\big)\|K\|_{\mathrm{Lip}}\big(|x-x'|+W^1(p_1,p_2)\big),\end{equation}
where $W^1$ denotes the Wasserstein-1 distance on $\mathcal{P}(\R^d\times\MM)$ or $\mathcal{P}(\MM)$ indifferently.

\subsubsection{The case of the von Mises distribution and application to the Vicsek and body-orientation models}

This general setting encompasses the two following examples derived from the Vicsek model. 
\begin{itemize} 
\item In the $d$-dimensional continuous Vicsek model described in the introduction, the Riemannian manifold is taken to be equal to the sphere $\Sph^{d-1}$, viewed as a submanifold of $\R^d$ endowed with its canonical Euclidean structure. The velocity map $\Phi$ is then the canonical injection $\Sph^{d-1}\hookrightarrow \R^d$.
\item In the three dimensional Body-Orientation model studied in \cite{toto}, the Riemannian manifold is taken to be equal to $SO_3(\R)$ viewed as a submanifold of $\MM_3(\R)$ endowed with the inner product $A\cdot B :=\frac{1}{2}\Tr(A^TB)$. The velocity map is the projection $A\in SO_3(\R)\mapsto Ae_1\in\R^3$ where $e_1$ is a fixed vector (first component of a reference frame). 
\end{itemize}

In both cases, in the previous works \cite{dimarcomotsch16,toto}, the interaction between the agents is given by a von Mises distribution. This family of probability laws has first been introduced for circular statistics on the circle and on the sphere \cite{mardia1975statistics}. It can be extended to a family of probability distributions on $SO_3(\R)$ (it is then sometimes referred as the von Mises--Fisher matrix distribution \cite{mardia2009directional,lee2018bayesian}) and more generally to a family of probability distributions on any compact embedded manifold. This defines an interaction law which satisfies the Assumptions \ref{asslocallybounded}, \ref{assumptionlocallipschitz}, \ref{assumptionfluxlipschitz} as shown in the following proposition. 

\begin{proposition} Let us define for $J\in E$ the von Mises distribution $M_J$ on $\MM$ by: 
\begin{equation}\label{vonmisesJ}
M_J(m):=\frac{\e^{J\cdot m}}{\mathcal{Z}}
\end{equation}
where $\mathcal{Z}$ is a normalisation constant. This defines an interaction law which satisfies Assumptions \ref{asslocallybounded}, \ref{assumptionlocallipschitz}, \ref{assumptionfluxlipschitz} with regularity constants
\[\alpha(a) = \e^{2a},\,\,\,L(a)=a\e^{2a}\,\,\,\text{and}\,\,\,\theta(a)=\e^{2a}+\e^{4a}.\]
\end{proposition}

\begin{proof} Let $a>0$ and $J,J'\in E$ be such that $|J|,|J'|\leq a$. 
\begin{enumerate}
\item For any $m\in\MM$ it holds that $|J\cdot m|\leq a$ so since the exponential function on $\R$ is non decreasing it holds that: 
\[\e^{J\cdot m}\leq \e^{a}\,\,\,\text{and}\,\,\,\mathcal{Z}:=\int_\MM \e^{J\cdot m'}\dd m'\geq \e^{-a}.\]
From which we deduce that
\begin{equation}\label{boundJZ}\|M_J\|_{L^\infty(\MM)}\leq \e^{2a}=:\alpha(a).\end{equation}
\item Let $m_1,m_2\in\MM$. Using the mean-value inequality on the compact segment $[\e^{-a},\e^a]$ and the fact that $\MM$ is isometrically embedded into $E$, we obtain that:
\[|M_J(m_1)-M_J(m_2)|\leq\e^{2a}|J\cdot m_1-J\cdot m_2|\leq a\e^{2a}d(m_1,m_2)=:L(a)d(m_1,m_2).\]
\item Let $m\in\MM$. It holds that: 
\[|M_J(m)-M_{J'}(m)|\leq\frac{1}{\mathcal{Z}}|\e^{J\cdot m}-\e^{J'\cdot m}|+\frac{\e^{J'\cdot m}}{\mathcal{Z}\mathcal{Z'}}|\mathcal{Z}-\mathcal{Z'}|\]
where
\[\mathcal{Z}:=\int_\MM e^{J\cdot m}\dd m\,\,\,\text{and}\,\,\,\mathcal{Z'}=\int_\MM \e^{J'\cdot m}\dd m.\]
Using again the mean-value inequality and the bound \eqref{boundJZ}, we obtain: 
\[|M_J(m)-M_{J'}(m)|\leq (\e^{2a}+\e^{4a})|J-J'|=:\theta(a)|J-J'|.\]
\end{enumerate}
\end{proof}

As described in the introduction, for the Vicsek and Body-orientation cases, an agent interacts with its neighbours by sampling a new orientation from a von Mises distribution, the parameter of which reflects the local average orientation of the other agents. Two cases have to be considered depending on how ``local'' is the interaction. 

\begin{enumerate}
\item The first case corresponds to the interaction at the level of the IBM: the density of the agents is given by a probability measure $\mu$ on $\R^d\times\MM$ (typically the empirical distribution of the agents) and the interaction law for an agent at position $x\in\R^d$ is the von Mises distribution \eqref{vonmisesJ} with parameter $J$ equal to the flux of the observation measure $K*\mu(x)$, namely 
\[M_J\equiv M_{J_{K*\mu}(x)}\equiv M_{K*\mu}[x].\]
In this case, the regularity constants in Assumptions \ref{asslocallybounded}, \ref{assumptionlocallipschitz}, \ref{assumptionfluxlipschitz} are functions of $\|K\|_{L^\infty}$ since $|J_{K*\mu}(x)|\leq\|K\|_{L^\infty}$. 

\item The second case corresponds to the interaction at the kinetic level when we let $K\to\delta_0$ : the density of agents is given by a probability density function $f$ (typically the solution of a BGK equation) which we assume to be bounded in the $L^\infty$ norm by a constant $a>0$. The interaction law at position $x\in\R^d$ is then the von Mises distribution \eqref{vonmisesJ} with parameter $J$ equal to the flux of the purely local observation measure $f(x,\cdot)$ : 
\[M_J\equiv M_{J_{f(x,\cdot)}}\equiv M_f[x].\]
\end{enumerate}
\subsection{Individual Based Model}\label{sectionIBM}

In the abstract framework described in Subsection \ref{sectionencompass}, the main focus of this article will be the study of the following IBM.\\

An agent $i\in\{1,\ldots,N\}$ is described at time $t$ by a couple $Z^{i,N}_t = (X^{i,N}_t,m^{i,N}_t)\in\R^d\times\MM$ of position and orientation. The evolution of the trajectories are given by the following Piecewise Deterministic Markov Process (PDMP) :  
\begin{enumerate}[$\bullet$] 
\item Let $(S_n)_n$ be a sequence of independent holding times which follow an exponential law of parameter $N$ (their expectation is $1/N$). The jump times are denoted by $T_n :=~S_1+~\ldots+S_n$. We set $S_0=T_0=0$. 
\item Let $(I_n)_n$ a sequence of independent indexes which follow a uniform law on $\{1,\ldots,N\}$.
\item Between two jump times on $[T_n,T_{n+1})$, the system evolves deterministically: 
\begin{equation}\label{deterministicpart}\forall t\in[T_n,T_{n+1}),\,\,\forall i\in\{1,\ldots,N\},\,\,\,\,
\left\{
\begin{array}{rcl}
X^{i,N}_t &=& X^{i,N}_{T_n}+(t-T_n)\Phi\big(m^{i,N}_t\big)\\[0.2cm]
m^{i,N}_t&=&m^{i,N}_{T_n}
\end{array}
\right.
.\end{equation}
\item At $T_{n+1}$ a jump occurs for the agent $I_n$ which draws a new orientation according to the interaction law: 
\begin{equation}\label{jumppart}m^{I_n,N}_{T_{n+1}} \sim M_{K*\hat{\mu}^N_{T_{n+1}^-}}\Big[X^{I_n,N}_{T_{n+1}^-}\Big],\end{equation}
where 
\[\hat{\mu}^N_t:=\frac{1}{N}\sum_{i=1}^N \delta_{(X^{i,N}_t,m^{i,N}_t)}\]
is the empirical measure at time $t$ and 
\[Z^{i,N}_{T_{n+1}^-} = \Big(X^{i,N}_{T_{n+1}^-},m^{i,N}_{T_{n+1}^-}\Big) := \Big(X^{i,N}_{T_n}+S_{n+1}\Phi\big(m^{i,N}_{T_n}\big),m^{i,N}_{T_n}\Big)\in\R^d\times\mathscr{M}.\]
\end{enumerate}

\subsection{Kinetic model: the BGK equation. Well-posedness results.}\label{sectionkinetic}

The kinetic model associated to the IBM described in Subsection \ref{sectionIBM} is given by the following BGK equation (see Theorem \ref{propagationofchaostheorem} below for more details on how the IBM is related to the kinetic model): 
\begin{equation}\label{bgkkernel}\tag{BGK-$K$}
\partial_t f+\Phi(m)\cdot\nabla_xf=\rho_f M_{K*f}[x]-f\end{equation}
where $f=f_t(\dd x,\dd m)$ is a time dependent probability measure on $\R^d\times\mathscr{M}$ and $\rho_{f_t}(\dd x)$ is the first marginal of $f_t$, i.e. $\rho_{f_t}(\dd x):=\int_{\mathscr{M}}f_t(\dd x,\dd m)$. For the sake of clarity, in the following we will write: 
\[G^K_{f_t}(\dd x,\dd m)\equiv \rho_{f_t}(\dd x)M_{K*f_t}[x](m)\dd m.\]
When $f_t$ has a density, i.e. $f_t\in\widetilde{\mathcal{P}}(\R^d\times \mathscr{M})$, we will write with a slight abuse of notations:
\[G^K_{f_t}(x,m)\equiv \rho_{f_t}(x)M_{K*f_t}[x](m)\]
which is a probability density function on $\R^d\times \mathscr{M}$. 
\begin{lemma}\label{GKbounds}
Let $f,g\in \widetilde{\mathcal{P}}(\R^d\times \MM)$, then it holds that
\begin{equation}\label{GKlipschitz}
\|G^K_f-G_g^K\|_{L^1(\R^d\times\MM)}\leq \Big(\alpha\big(\|K\|_{L^\infty}\big)+\|K\|_{L^\infty}\theta\big(\|K\|_{L^\infty}\big)\Big)\|f-g\|_{L^1(\R^d\times\MM)}
\end{equation}
\end{lemma}

\begin{proof} Since the $G^K_f$ and $G^K_g$ are a product of two quantities, the $L^1$ norm of the difference can be split in two parts:
\begin{align*}
&\iint_{\R^d\times \mathscr{M}} |G^K_{f}-G^K_{g}|(x,m)\,\dd x\dd m = \iint_{\R^d\times \MM} |\rho_f(x)M_{K*f}[x](m)-\rho_g(x)M_{K*g}[x](m)|\dd x\dd m \\
&\leq \iint_{\R^d\times\MM} |\rho_f(x)-\rho_g(x)|M_{K*f}[x](m)\dd x\dd m \\
&\hspace{3cm}+ \iint_{\R^d\times\MM} \rho_g(x)|M_{K*f}[x](m)-M_{K*g}[x](m)|\dd x\dd m.
\end{align*}
Since for all $x\in\R^d$, $|J_{K*f}(x)|\leq \|K\|_{L^\infty}$, the first integral on the right-hand side can be bounded by 
\[\alpha(\|K\|_{L^\infty})\iint_{\R^d\times \mathscr{M}} |\rho_{f}-\rho_{g}|\dd x \dd m\]
using Assumption \ref{asslocallybounded}. For the second integral on the right-hand side, using \eqref{fluxlipschitz}, it holds that: 
\begin{multline*}\iint_{\R^d\times\MM} \rho_g(x)|M_{K*f}[x](m)-M_{K*g}[x](m)|\dd x\dd m\\ 
\leq \theta(\|K\|_{L^\infty})\iint_{\R^d\times\MM}\rho_g(x) |J_{K*f}(x)-J_{K*g}(x)|\dd x\dd m\\=\theta(\|K\|_{L^\infty})\int_{\R^d}\rho_g(x) |J_{K*f}(x)-J_{K*g}(x)|\dd x\end{multline*}
where we have used that $\int_\MM \dd m=1$. Then, we note that $|m|\leq 1$ so the previous bound gives: 
\begin{multline*}\iint_{\R^d\times\MM} \rho_g(x)|M_{K*f}[x](m)-M_{K*g}[x](m)|\dd x\dd m\\ \leq \theta(\|K\|_{L^\infty})\iint_{\R^d\times \MM} \rho_g(x)|K*f(x,m)-K*g(x,m)|\dd x\dd m.\end{multline*}
By Definition \ref{observationmeasure} of the observation measure, we obtain:
\begin{multline*}\iint_{\R^d\times\MM} \rho_g(x)|M_{K*f}[x](m)-M_{K*g}[x](m)|\dd x\dd m\\ \leq \theta(\|K\|_{L^\infty})\iint_{\R^d\times \MM} \rho_g(x)K(x-y)|f(y,m)-g(y,m)|\dd x\dd y\dd m.\end{multline*}
 To conclude, we use the fact that $g$ is a probability measure so that the last integral on the right-hand side can be bounded by $\|K\|_{L^\infty}\|f-g\|_{L^1(\R^d\times\MM)}$. The result follows.
 \end{proof}

A mild-solution to \eqref{bgkkernel} is defined as an element $f\in C\big([0,T],\mathcal{P}(\R^d\times\mathscr{M})\big)$ which satisfies for all $\varphi\in C_b(\R^d\times\mathscr{M})$: 
\begin{equation}\label{mildsolution}\langle\varphi,f_t\rangle=\e^{-t}\langle\varphi,\mathsf{T}_tf_0\rangle+\int_0^t \e^{-(t-s)}\langle \varphi,\mathsf{T}_{t-s}G^K_{f_s}\rangle\,\dd s,\end{equation}
where $\mathsf{T}_t$ is the free-transport operator: 
\begin{equation}\label{freetransportoperator}\langle \varphi,\mathsf{T}_t\mu\rangle:=\iint_{\R^d\times \mathscr{M}}\varphi(x+t\Phi(m),m)\mu(\dd x,\dd m).\end{equation}
Similarly, in the degenerate case $K=\delta_0$, we consider
\begin{equation}\label{bgkdelta}\tag{BGK-$\delta$}
\partial_t f+(\Phi(m)\cdot\nabla_x)f=\rho_f M_{f}[x]-f\end{equation}
where $f=f_t(x,m)$ is a time dependent probability density in $L^\infty(\R^d\times\mathscr{M})$ and $\rho_{f_t}(\dd x)$ is the first marginal of $f_t$. For the sake of clarity, in the following we will write: 
\[G_{f_t}(\dd x,\dd m)\equiv \rho_{f_t}(\dd x)M_{f_t}[x](\dd m).\]

\begin{lemma}
Let $a>0$ and let $f,g\in L^\infty\cap\widetilde{\mathcal{P}}(\R^d\times\MM)$ be two probability density functions such that $\|f\|_{L^\infty},\|g\|_{L^\infty}\leq a$. Then 
\begin{equation}\label{Glipschitzinfty}
\|G_f-G_g\|_{L^\infty(\R^d\times\MM)}\leq (\alpha(a)+a\theta(a))\|f-g\|_{L^\infty(\R^d\times \MM)}\end{equation}
and
\begin{equation}\label{Glipschitzl1}
\|G_f-G_g\|_{L^1(\R^d\times\MM)}\leq (\alpha(a)+a\theta(a))\|f-g\|_{L^1(\R^d\times \MM)}\end{equation}
\end{lemma}

\begin{proof}
The result follows as before from \eqref{assumptionlocallybounded} and \eqref{fluxlipschitz} by noticing that for all $x\in\R^d$, $|J_f(x)|\leq a$ and $|J_f(x)-J_g(x)|\leq \|f-g\|_{L^\infty}$.
\end{proof}

The well-posedness of \eqref{bgkkernel} and \eqref{bgkdelta} is given by the two following propositions. They are based on Duhamel's formula and on a fixed point argument.

\begin{proposition}[Well-posedness of \eqref{bgkkernel}]\label{wellposednesskernel}
For all $f_0\in \widetilde{\mathcal{P}}(\R^d\times \mathscr{M})$ and all $T>0$, there exists a unique solution of \eqref{bgkkernel} in $C\big([0,T],\widetilde{\mathcal{P}}(\R^d\times \mathscr{M})\big)$ with initial condition $f_0$. Moreover if $t\mapsto f^1_t$ and $t\mapsto f^2_t$ are two solutions of \eqref{bgkkernel} with respective initial conditions $f^1_0\in \widetilde{\mathcal{P}}(\R^d\times \mathscr{M})$ and $f^2_0\in \widetilde{\mathcal{P}}(\R^d\times \mathscr{M})$ then 
\[\sup_{t\in[0,T]}\|f^1_t-f^2_t\|_{L^1(\R^d\times \mathscr{M})}\leq\|f^1_0-f^2_0\|_{L^1(\R^d\times \mathscr{M})}\e^{\big(\alpha\big(\|K\|_{L^\infty}\big)+\|K\|_{L^\infty}\theta\big(\|K\|_{L^\infty}\big)\big)T}.\]
\end{proposition}

\begin{proof}
Let $f_0\in \widetilde{\mathcal{P}}(\R^d\times \mathscr{M})$ and $T>0$ and let us define the map: 
\[\mathcal{T}:C\big([0,T],\widetilde{\mathcal{P}}(\R^d\times \mathscr{M})\big)\to C\big([0,T],\widetilde{\mathcal{P}}(\R^d\times \mathscr{M})\big)\]
by 
\[\mathcal{T}(g)(t,x,m):=\e^{-t}f_0(x-t\Phi(m),m)+\int_0^t \e^{-(t-s)}G^K_{g_s}\big(x-(t-s)\Phi(m),m\big)\dd s.\]
We prove that $\mathcal{T}$ is a contraction. Let $g^1,g^2\in C\big([0,T],\widetilde{\mathcal{P}}(\R^d\times \mathscr{M})\big)$. One has: 
\[\iint_{\R^d\times \mathscr{M}} |\mathcal{T}(g^1)-\mathcal{T}(g^2)|(t,x,m)\,\dd x\dd m\leq \int_0^t \e^{-(t-s)} \iint_{\R^d\times \mathscr{M}} |G^K_{g^1_s}-G^K_{g^2_s}|(x,m)\,\dd x\dd m\,\dd s.\]
The Lipschitz bound \eqref{GKlipschitz} leads to:
\[\sup_{t\in[0,T]}\|\mathcal{T}(g^1)(t)-\mathcal{T}(g^2)(t)\|_{L^1}\leq \Big(\alpha\big(\|K\|_{L^\infty}\big)+\|K\|_{L^\infty}\theta\big(\|K\|_{L^\infty}\big)\Big)\int_0^T \sup_{u\in[0,s]}\|g^1_u-g^2_u\|_{L^1}\,\dd s,\]
which proves that an iteration of $\mathcal{T}$ is a contraction and therefore $\mathcal{T}$ has a unique fixed point. This fixed point is a (mild) solution of \eqref{bgkkernel}. The stability estimate follows similarly by Gronwall lemma. 
\end{proof}

\begin{rem}
The well-posedness result Proposition \ref{wellposednesskernel} is stated in an absolutely continuous framework (i.e. for probability density functions rather than probability measures) and a solution is defined as a fixed point of the map $\mathcal{T}$. The associated probability measure (i.e $f_t(x,m)\dd x\dd m$) is a mild solution of \eqref{bgkkernel} as defined by \eqref{mildsolution}. However, a straightforward adaptation of this proof actually gives the well-posedness in $C\big([0,T],{\mathcal{P}}(\R^d\times \mathscr{M})\big)$ where ${\mathcal{P}}(\R^d\times \mathscr{M})$ is endowed with the total variation norm (which is the $L^1$ norm for probability density functions). The result is then similar to the one obtained in \cite[Proposition 2]{monmarche} (by probabilistic coupling arguments) or in \cite{canizo2018asymptotic} (by Duhamel's formula and a fixed point argument). In the following, only the absolutely continuous framework will be considered. 
\end{rem}

For the well-posedness of \eqref{bgkdelta}, we consider the space $L^\infty\cap\widetilde{\mathcal{P}}(\R^d\times \mathscr{M})$ which is a closed subspace of $L^\infty\cap L^1(\R^d\times \MM)$ and therefore complete for the norm $\|\cdot\|_{L^\infty}+\|\cdot\|_{L^1}$ which will be denoted for short by $\|\cdot\|_{L^1\cap L^\infty}$. 

\begin{proposition}[Well-posedness of \eqref{bgkdelta}]\label{wellposednessdelta} Let $f_0\in L^\infty\cap\widetilde{\mathcal{P}}(\R^d\times \mathscr{M})$ such that $\|f_0\|_{L^\infty}< a$ for a given $a>0$. Then there exists a time $T>0$ and a unique solution of \eqref{bgkdelta} in $C\big([0,T],L^\infty\cap\widetilde{\mathcal{P}}(\R^d\times \mathscr{M})\big)$ with initial condition $f_0$. Moreover if $t\mapsto f^1_t$ and $t\mapsto f^2_t$ are two solutions of \eqref{bgkdelta} with respective initial conditions $f^1_0$ and $f^2_0$ such that $\|f^1_0\|_{L^\infty},\|f^2_0\|_{L^\infty}< a$ for the same $a>0$ then 
\[\sup_{t\in[0,T]}\|f^1_t-f^2_t\|_{L^1\cap L^\infty}\leq\|f^1_0-f^2_0\|_{L^1\cap L^\infty}\e^{c(T)}\]
where $c(T)$ is a constant which depends only on $T$. 
\end{proposition}

\begin{proof}
Let $f_0\in L^\infty\cap\widetilde{\mathcal{P}}(\R^d\times \mathscr{M})$. Let $a>0$ such that $\|f_0\|_{L^\infty}< a$. Let $B_a$ the ball of radius $a$ in $C\Big([0,T],L^\infty\cap\widetilde{\mathcal{P}}(\R^d\times \mathscr{M})\Big)$ where $T>0$ will be specified later and for the norm $\|h\|=\sup_{t\in[0,T]}\|h(t)\|_{L^\infty}$. We consider the map: 
\[\mathcal{T}:B_a\to B_a\]
defined by: 
\[\mathcal{T}(g)(t,x,m):=\e^{-t}f_0(x-t\Phi(m),m)+\int_0^t \e^{-(t-s)}G_{g_s}\big(x-(t-s)\Phi(m),m\big)\dd s.\]
Using the bound \eqref{assumptionlocallybounded}, one has for $g\in B_a$: 
\[|\mathcal{T}(g)(t,x,m)|\leq \e^{-t}\|f_0\|_{L^\infty}+(1-\e^{-t})a\alpha(a)\]
and the map $\mathcal{T}$ is therefore well defined for $T>0$ small enough to ensure that for all $t\in[0,T]$: 
\[\e^{-t}\|f_0\|_{L^\infty}+(1-\e^{-t})a\alpha(a)\leq a.\]
Namely $\mathcal{T}$ is well defined for 
\[T\leq\log\left(\frac{a\alpha(a)-a}{a\alpha(a)-\|f_0\|_{L^\infty}}\right)\]
where we can assume without loss of generality that $\alpha(a)\geq 1$. We prove that $\mathcal{T}$ is a contraction. Let $g^1,g^2\in B_a$. One has: 
\[|\mathcal{T}(g^1)-\mathcal{T}(g^2)|(t,x,m)\leq\int_0^t \e^{-(t-s)} \big|G_{g^1_s}-G_{g^2_s}\big|\big(x-(t-s)\Phi(m),m\big)\dd s.\]
Using the Lipschitz bounds \eqref{Glipschitzinfty} and \eqref{Glipschitzl1} one has: 
\[\sup_{t\in[0,T]}\|\mathcal{T}(g^1)(t)-\mathcal{T}(g^2)(t)\|_{L^\infty}\leq (\alpha(a)+a\theta(a))\int_0^T \|g^1_s-g^2_s\|_{L^\infty}\dd s,\]
and
\[\sup_{t\in[0,T]}\|\mathcal{T}(g^1)(t)-\mathcal{T}(g^2)(t)\|_{L^1}\leq (\alpha(a)+a\theta(a))\int_0^T \|g^1_s-g^2_s\|_{L^1}\dd s,\]
and therefore, an iteration of $\mathcal{T}$ is a contraction and therefore $\mathcal{T}$ has a unique fixed point which is a solution of \eqref{bgkdelta}. The stability estimate follows similarly by Gronwall lemma. 
\end{proof}

\subsection{Main results}\label{sectionmainresults}

Let  $Z^{i,N}_t=(X^{i,N}_t,m^{i,N}_t)\in\R^d\times\MM$, $i\in\{1,\ldots,N\}$ be a system of PDMP defined by the IBM described in Subsection \ref{sectionIBM}. 

\begin{theorem}[Propagation of chaos]\label{propagationofchaostheorem} Let $f_0\in \widetilde{\mathcal{P}}(\R^d\times\MM)$. Let $f_t$ be the solution of \eqref{bgkkernel} at time $t>0$ with initial condition $f_0$. Assume that initially the agents are independent and identically distributed with respect to the law $f_0$.
\begin{enumerate}
\item For all $t>0$, it holds that:
\begin{equation}\label{propagationofchaoshatmu}\E[W^1(\hat{\mu}^N_t,f_t)]\underset{N\to+\infty}{\longrightarrow}0\end{equation}
where 
\[\hat{\mu}^N_t:=\frac{1}{N}\sum_{i=1}^N \delta_{(X^{i,N}_t,m^{i,N}_t)}.\]
\item Let $\mu^N_t\in\mathcal{P}(\R^d\times\mathscr{M})$ be the law at time $t>0$ of any agent. Then for all $t>0$ it holds that:
\begin{equation}\label{propagationofchaos}W^1(\mu^N_t,f_t)\leq C\frac{\e^{\left(2\lambda+\frac{\sigma(K)}{N}\right)t}  \sqrt{\|K\|_{L^\infty}}}{\|K\|_{\mathrm{Lip}}\sqrt{N}}\exp\left(t\sigma(K)\e^{\frac{\sigma(K)}{N}}\right),\end{equation}
where $C>0$ is a constant which depends only on $\MM$ and $d$ and 
\[\sigma(K) :=2\theta\big(\|K\|_{L^\infty}\big)\|K\|_{\mathrm{Lip}}.\]
\end{enumerate}
\end{theorem}

The convergence result \eqref{propagationofchaoshatmu} and  \cite[Proposition 2.2]{sznitman89}  imply the $f_t$-chaoticity in the sense of \cite[Definition 2.1]{sznitman89} of the law (in $\mathcal{P}((\R^d\times \MM)^N)$) of the system $(Z^{i,N}_t)_t$, $i\in\{1,\ldots,N\}$ as stated in the following corollary.

\begin{corollary}
Let $k\in\N$ and let $\mu_t^{N,k}$ denote the law of any $k$ agents of the system $(Z^{i,N}_t)_t$, $i\in\{1,\ldots,N\}$ at time $t$. Assume that initially the agents are independent and identically distributed with respect to the law $f_0$.  Then for every $k$-tuple of bounded continuous functions $\phi_1,\ldots\phi_k$ on $\R^d\times\mathscr{M}$ it holds that: 
\begin{multline*}\int_{(\R^d\times\mathscr{M})^k} \phi_1(x_1,m_1)\ldots\phi_k(x_k,m_k)\dd\mu^{N,k}_t(x_1,m_1,\ldots,x_k,m_k)\\
\underset{N\to+\infty}{\longrightarrow}\left(\int_{\R^d\times\mathscr{M}}\phi_1(x,m)f_t(x,m)\dd x\dd m\right)\times\ldots\times\left(\int_{\R^d\times\mathscr{M}}\phi_k(x,m)f_t(x,m)\dd x\dd m\right)\end{multline*}
where $f_t$ is the solution of \eqref{bgkkernel} at time $t>0$ with initial condition $f_0$. 
\end{corollary}

The bound \eqref{propagationofchaos} is similar to the bound obtained in \cite[Proposition 2.3]{jourdainmeleard} for the classical McKean-Vlasov system. Moreover it is explicit in terms of the norm of $K$. In particular, it is possible to take a kernel which depends on the number of agents $N$. In the following, we will consider a sequence of kernels: 
\[K^N(x):=\frac{1}{\varepsilon_N^d}K\left(\frac{x}{\varepsilon_N}\right),\]
where $\varepsilon_N\to0$ and $K$ is a smooth observation kernel. Without loss of generality, we can assume that the infinite and Lipschitz norms of $K$ are equal to 1. In particular the sequence $(K^N)_N$ is an approximation of the unity (meaning that $K^N\to\delta_0$ as $N\to+\infty$) and the bound \eqref{propagationofchaos} is still relevant provided that $\varepsilon_N\to0$ slowly enough. This type of interaction is called \textit{moderate} following the terminology of \cite{jourdainmeleard,Oelschl_ger_1985}. Based on this explicit bound and on regularity results for the sequence of solutions associated to the sequence of kernels (see in particular Lemma \ref{lemmapropertiesfN}), we will prove the following theorem.  

\begin{theorem}[Moderate interaction]\label{moderateinteractiontheorem}
Let $a>0$ and $T>0$ such that \eqref{bgkdelta} is well posed on $C([0,T],B_a)$ where $B_a$ is the ball of radius $a>0$ in $L^\infty\cap\widetilde{\mathcal{P}}(\R^d\times\MM)$. Let $f_0\in B_a$. Let us define the sequence of rescaled interaction kernels: 
\[K^N(x):=\frac{1}{\varepsilon_N^d}K\left(\frac{x}{\varepsilon_N}\right)\]
where $\varepsilon_N\to0$ slowly enough so that: 
\begin{equation}\label{boundepsilonN}
\frac{\exp\left(2T\theta(\varepsilon_N^{-d})\varepsilon_{N}^{-(d+1)}\right)}{\sqrt{N}}\underset{N\to+\infty}{\longrightarrow}0.\end{equation}
Let $\mu^N_t$ be the law at time $t<T$ of any agent defined by the PDMP \eqref{deterministicpart}, \eqref{jumppart} with the interaction kernel $K^N$.Then  it holds that:
\[\mu^N_t\overset{*}{\underset{N\to+\infty}{\longrightarrow}} f_t,\]
where $f_t$ is the solution of \eqref{bgkdelta} at time $t<T$ and where the convergence is the weak convergence of measures.
\end{theorem}

\begin{rem}
In the Vicsek and Body-orientation cases, one has $\theta(\varepsilon_N^{-d})=\e^{c/\varepsilon_N^{-d}}$. In order to fulfill Hypothesis \eqref{boundepsilonN}, we have to take $\varepsilon_N\sim \log\log(N)^{-1/d}$. From a physical point of view, this is not a satisfactory order of magnitude even for very large values of $N$. Moreover $\varepsilon_N$ decays much slower than $N^{-\alpha}$ with $\alpha\in(0,1)$ which is the result obtained in \cite{Oelschl_ger_1985} in a much simpler setting with constant diffusion. Similarly as in \cite{jourdainmeleard}, the estimate obtained here should therefore be understood as a purely theoretical result needed to obtain rigorously the purely local BGK equation and is not claimed to be optimal. 
\end{rem}

\begin{rem}
Both theorems rely on the fact that, despite its nonlinearity, the interaction is regular enough (Lipschitz assumptions) to prove the well-posedness of the kinetic PDEs. As mentioned in the introduction we could also consider another version of this model where the flux is ``normalised'' which consists in taking instead of a flux $J\in E$ its projection on the manifold $\MM$. This case is much more singular and is left for future work. See however the approach of \cite{figalli2018global,gamba2016global} for well-posedness results on the ``normalised" Vicsek model. 
\end{rem}

\section{Propagation of chaos (proof of Theorem \ref{propagationofchaostheorem})}\label{propagationofchaossection}

This section is devoted to the proof of Theorem \ref{propagationofchaostheorem}. It is based on a coupling argument as in \cite{sznitman89}. More precisely we define $N$ independent so-called McKean processes such that their common law satisfies \eqref{bgkkernel} (also known as McKean-Vlasov processes or Distribution Dependent processes in the literature for diffusion models \cite{eberle2018quantitative,veretennikov2006ergodic,wang2018distribution,mishura2016existence}). We then use a coupling between each one of the $N$ particles described in Subsection \ref{sectionIBM} with each one of the $N$ McKean processes to control the expectation of the distance between their paths over time. To obtain this control, firstly, we couple the time of the jumps between the processes and secondly, we use an optimal coupling between the laws of the new orientations of the particle processes and the McKean processes to control the expectation of the distance between the two paths at a jump. At each jump time, the expectation of the distance between a particle process and its associated McKean process is bounded by the sum of the average distance between each pair of particle and McKean processes just after the previous jump and an error term which tends to zero when $N$ tends to $+\infty$.  The control of the expectation of the distance between their paths over time then follows from a discrete Gronwall type inequality. 

\begin{proof}[\sc Proof (of Theorem \ref{propagationofchaostheorem})]

We first define $N$ independent copies of a so-called McKean process $\overline{Z}_t$ with which the processes on the agents will be coupled. \\

\noindent\textbf{Step 1.} \textit{McKean processes and coupling}\\

The McKean process is the (non-homogemeous) PDMP defined by:
\begin{enumerate}[$\bullet$]
\item a homogeneous Poisson process of rate 1, with holding times $(\overline{S}_n)_n$ and jumping times 
\[\overline{T}_n:=\overline{S}_1+\ldots+\overline{S}_{n},\]
\item the deterministic flow $\phi_t :\R^d\times \MM\to\R^d\times\MM$ for $t\geq0$:
\[\phi_t(x,m)=(x-t\Phi(m),m),\]
\item the transition probability at time $\overline{T}_n$: 
\[Q_{\overline{T}_n}(x,m):=\delta_x\otimes M_{K*f_{\overline{T}_n^-}}[x],\]
where $f_{\overline{T}_n^-}:=\mathsf{T}_{\overline{S}_n}f_{\overline{T}_{n-1}}=f_{\overline{T}_{n-1}}\circ\phi_{\overline{S}_n}$ and $f_t$ is the solution at time $t$ of \eqref{bgkkernel}.
\end{enumerate}

The well definition of the McKean process follows from the wellposedness of \eqref{bgkkernel} (Proposition \ref{wellposednesskernel}) as in \cite[Theorem 1.1]{sznitman89}. It is defined in such a way that its law satisfies \eqref{bgkkernel}. \\ 

Now for $i\in\{1,\ldots,N\}$ let $(\overline{Z}^i_t)_t$ be an independent copy of the McKean process. The positions and orientation components are denoted by $\overline{Z}^i_t=(\overline{X}^i_t,\overline{m}^i_t)$. Let us define the holding times $(S_n)_n$ and jumping times $(T_n)_n$ resulting from merging the Poisson processes associated to the $N$ independent copies of the McKean process. It defines a homogeneous Poisson process of rate $N$ \cite[Section 8.1.3]{bremaud}. In particular, the holding times $S_n$ are independent and follow an exponential law of parameter $N$ (their expectation is $1/N$). Let us also define $(I_n)_n$ the sequence of indexes in $\{1,\ldots,N\}$ such that the process $I_n$ is responsible for the $n$-th jump. The $I_n$ form a sequence of i.i.d. uniform random variables on $\{1,\ldots,N\}$.\\

Let us define a system $N$ PDMP $(Z^i_t)_t$, $i\in\{1,\ldots,N\}$, on $\R^d\times\MM$, with components denoted by $Z^i_t=(X^i_t,m^i_t)$, as follows.

\begin{enumerate}[$\bullet$]
\item Initially for all $i\in\{1,\ldots,\}$, $\overline{Z}^i_0= Z^{i}_0$. 
\item The jump times $(T_n)$ and the indexes $(I_n)$ are given above. 
\item The processes follow the deterministic flow $\phi$ between two jump times. 
\item At time $T_n$, the new orientation $m^{I_n}_{T_n}$ of $Z^{I_n}_{T_n}$ is defined by: 
\[m^{I_n}_{T_n} := s\left(\overline{m}^{I_n}_{T_n}\right),\]
where $\overline{m}^{I_n}_{T_n}$ is the orientation of $\overline{Z}^{I_n}_{T_n}$ at time $T_n$ and $s:\MM\to\MM$ is a Borel map such that:
\[s^\sharp M_{K*f_{T_n^-}}\left[\overline{X}^{I_n}_{T_n^-}\right]=M_{K*\hat{\mu}_{T_n^-}}\left[X^{I_n}_{T_n^-}\right]\]
and $s$ is an optimal transport map between the two distributions. In particular it implies that the random variable $m^{I_n}_{T_n}$ is distributed according to $M_{K*\hat{\mu}_{T_n^-}}\left[X^{I_n}_{T_n^-}\right]$ (conditionally to $\mathcal{F}_{T_n^-}$) and
\begin{equation}\label{optimal}\E\left[d\left(m^{I_n}_{T_n},\overline{m}^{I_n}_{T_n}\right)\Big|\mathcal{F}_{T_n^-}\right]=W^1\left(M_{K*f_{T_n^-}}\left[\overline{X}^{I_n}_{T_n^-}\right],M_{K*\hat{\mu}_{T_n^-}}\left[X^{I_n}_{T_n^-}\right]\right)\end{equation}
\end{enumerate}
where the $\sigma$-algebra $\mathcal{F}_{T_n^-}$ is defined below.  \\

The existence of such optimal transport map is given by \cite[Theorem 1]{feldman2002monge} (Monge problem, see \cite{villanioptimaltransport}) and unicity can be recovered under additional assumptions which will not be needed here.

\begin{rem}In \cite[Section 3.3]{degond2018alignment} an explicit transport map is used : in the $SO_3(\R)$-framework a transport map from a von Mises distribution of parameter $\Lambda_1\in SO_3(\R)$ to a von Mises of parameter $\Lambda_2\in SO_3(\R)$ is given by 
\[A\mapsto \Lambda_1^T A\Lambda_2.\]

\end{rem}

We define the following filtrations: 
\[\mathcal{G}_n := \sigma(S_1,\ldots,S_n),\]
\[\mathcal{F}_t := \sigma\Big(Z^{i}_s,\overline{Z}^i_s \,|\, 0\leq s \leq t, \,i\in\{1,\ldots,N\}\Big),\]
and we will write: 
\[\mathcal{F}_{T_n^-} :=\mathcal{F}_{T_{n-1}}\vee\sigma(S_n)\vee\sigma(I_n).\]
In particular, 
\[\mathcal{G}_n\subset\mathcal{F}_{T_n^-}.\]

\noindent\textbf{Step 2.} \textit{Control of the jumps}\\

In this step, we bound the expectation of the distance between the processes at the jump time $T_n$, knowing the system at time $T_{n-1}$ and the holding time $S_n$. By definition of the processes, the couples of position and orientation
\[Z^{I_n}_{T_n} = (X^{I_n}_{T_n},m^{I_n}_{T_n})\,\,\,\text{and}\,\,\,\overline{Z}^{I_n}_{T_n} = (\overline{X}^{I_n}_{T_n},\overline{m}^{I_n}_{T_n})\]
of the particle $I_n$ and its associated McKean process satisfy:
\begin{align}\label{expectationdistance}
\E\left[d\Big(Z^{I_n}_{T_n},\overline{Z}^{I_n}_{T_n}\Big) \Big| \mathcal{F}_{T_n^-}\right]
&\leq \big|X^{I_n}_{T_n^-}-\overline{X}^{I_n}_{T_n^-}\big|+ W^1\Big(M_{K*\hat{\mu}^N_{T_n^-}}\big[X^{I_n}_{T_n^-}\big],M_{K*f_{T_n^-}}\big[\overline{X}^{I_n}_{T_n^-}\big]\Big)\nonumber\\
&\leq(1+\lambda S_n)d\Big(Z^{I_n}_{T_{n-1}},\overline{Z}^{I_n}_{T_{n-1}}\Big) + W^1\Big(M_{K*\hat{\mu}^N_{T_n^-}}\big[X^{I_n}_{T_n^-}\big],M_{K*f_{T_n^-}}\big[\overline{X}^{I_n}_{T_n^-}\big]\Big)
\end{align}
where we used \eqref{optimal} in the first inequality and the fact that the flow is $\lambda$-Lipschitz in the second inequality. \\

The last $W^1$-distance is split in three parts.
\begin{multline}\label{threeterms}
 W^1\Big(M_{K*\hat{\mu}^N_{T_n^-}}\big[X^{I_n}_{T_n^-}\big],M_{K*f_{T_n^-}}\big[\overline{X}^{I_n}_{T_n^-}\big]\Big) \\
\hspace{-1cm} \leq  W^1\Big(M_{K*\hat{\mu}^N_{T_n^-}}\big[X^{I_n}_{T_n^-}\big],M_{K*\hat{\mu}^N_{T_n^-}}\big[\overline{X}^{I_n}_{T_n^-}\big]\Big)\\
\hspace{3cm}+ W^1\Big(M_{K*\hat{\mu}^N_{T_n^-}}\big[\overline{X}^{I_n}_{T_n^-}\big],M_{K*\overline{\mu}^N_{T_n^-}}\big[\overline{X}^{I_n}_{T_n^-}\big]\Big)\\
+ W^1\Big(M_{K*\overline{\mu}^N_{T_n^-}}\big[\overline{X}^{I_n}_{T_n^-}\big],M_{K*f_{T_n^-}}\big[\overline{X}^{I_n}_{T_n^-}\big]\Big)
\end{multline}

Where $\overline{\mu}^N_t$ is the empirical measure of the nonlinear processes: 
\[\overline{\mu}^N_t := \frac{1}{N}\sum_{i=1}^N \delta_{\overline{Z}^i_t}.\]
For the first term on the right-hand side of \eqref{threeterms}, we use \eqref{wassersteinboundM} and the fact that the dynamics is deterministic on the time interval $[T_{n-1},T_n)$, this leads to the following estimate:
\begin{align*}
W^1\Big(M_{K*\hat{\mu}^N_{T_n^-}}\big[X^{I_n}_{T_n^-}\big],M_{K*\hat{\mu}^N_{T_n^-}}\big[\overline{X}^{I_n}_{T_n^-}\big]\Big)&\leq \theta\big(\|K\|_{L^\infty}\big)\|K\|_{\mathrm{Lip}}\big|X^{I_n}_{T_n^-}-\overline{X}^{I_n}_{T_n^-}\big|\\
&\leq \theta\big(\|K\|_{L^\infty}\big)\|K\|_{\mathrm{Lip}}(1+\lambda S_n)d\Big(Z^{I_n}_{T_{n-1}},\overline{Z}^{I_n}_{T_{n-1}}\Big).
\end{align*}

Similarly for the second term on the right-hand side of \eqref{threeterms}:
\begin{align*}W^1\Big(M_{K*\hat{\mu}^N_{T_n^-}}\big[\overline{X}^{I_n}_{T_n^-}\big],M_{K*\overline{\mu}^N_{T_n^-}}\big[\overline{X}^{I_n}_{T_n^-}\big]\Big) &\leq \theta\big(\|K\|_{L^\infty})\|K\|_{\mathrm{Lip}}W^1(\hat{\mu}^N_{T_n^-},\overline{\mu}^N_{T_n^-})\\
&\leq \theta\big(\|K\|_{L^\infty})\|K\|_{\mathrm{Lip}}\frac{1}{N}\sum_{i=1}^N d\Big(Z^{i}_{T_n^-},\overline{Z}^i_{T_n^-}\Big)\\
&\leq \theta\big(\|K\|_{L^\infty})\|K\|_{\mathrm{Lip}}(1+\lambda S_n)\frac{1}{N}\sum_{i=1}^N d\Big(Z^{i}_{T_{n-1}},\overline{Z}^i_{T_{n-1}}\Big).
\end{align*}

The third term on the right-hand side of \eqref{threeterms} does not involve the processes $Z^i_t$ and will be considered as an \textit{error term}. Using \eqref{fluxlipschitz}, it holds that
\begin{align*}W^1\Big(M_{K*\overline{\mu}^N_{T_n^-}}\big[\overline{X}^{I_n}_{T_n^-}\big],M_{K*f_{T_n^-}}\big[\overline{X}^{I_n}_{T_n^-}\big]\Big)&\leq \theta\big(\|K\|_{L^\infty})\Big|J_{K*\overline{\mu}^N_{T_n^-}}\big(\overline{X}^{I_n}_{T_n^-}\big)-J_{K*f_{T_n^-}}\big(\overline{X}^{I_n}_{T_n^-}\big)\Big|\\
&=:\theta\big(\|K\|_{L^\infty})e^{I_n}_{T_n^-}.\end{align*}

In addition, for all $i\ne I_n$, it holds that
\[\E\left[d\Big(Z^{i}_{T_n},\overline{Z}^{i}_{T_n}\Big) \Big| \mathcal{F}_{T_n^-}\right]\leq (1+\lambda S_n)d\Big(Z^{i}_{T_n},\overline{Z}^{i}_{T_n}\Big).\]

Gathering everything and from \eqref{expectationdistance} it leads to
\begin{multline*}
\E[Y_{T_n}|\mathcal{F}_{T_n^-}]\leq (1+\lambda S_n)Y_{T_{n-1}}\\
\hspace{-2cm}+\theta\big(\|K\|_{L^\infty}\big)\|K\|_{\mathrm{Lip}}(1+\lambda S_n)d\Big(Z^{I_n}_{T_{n-1}},\overline{Z}^{I_n}_{T_{n-1}}\Big)\\
+\theta\big(\|K\|_{L^\infty})\|K\|_{\mathrm{Lip}}(1+\lambda S_n)\frac{1}{N}\sum_{i=1}^N d\Big(Z^{i}_{T_{n-1}},\overline{Z}^i_{T_{n-1}}\Big)+\theta\big(\|K\|_{L^\infty})e^{I_n}_{T_n^-}
\end{multline*}
where 
\[Y_t := \sum_{i=1}^N d\Big(Z^i_t,\overline{Z}^i_t\Big).\]
Taking the conditional expectation with respect to $\mathcal{G}_n$ gives:
\begin{multline}\label{VI}\E[Y_{T_n}|\mathcal{G}_n]\leq(1+\lambda S_n)\E[Y_{T_{n-1}}|\mathcal{G}_{n-1}]+(1+\lambda S_n)\frac{\sigma(K)}{N}\E[Y_{T_{n-1}}|\mathcal{G}_{n-1}]\\+\theta\big(\|K\|_{L^\infty})\E\Big[e^{I_n}_{T_n^-}\Big|\mathcal{G}_n\Big],\end{multline}
where we have used that $S_n$ is $\mathcal{G}_n$-measurable and the independence relations : $I_n\perp\mathcal{F}_{T_{n-1}}$, $I_n\perp \mathcal{G}_n$ and $I_n\sim\mathcal{U}({1,\ldots,N})$ which imply that
\[\E\left[d\left(Z^{I_n}_{T_{n-1}},\overline{Z}^{I_n}_{T_{n-1}}\right)\Big|\mathcal{G}_{n}\right]=\frac{1}{N}\E\left[Y_{T_{n-1}}|\mathcal{G}_n\right],\]
and since $\mathcal{G}_n=\mathcal{G}_{n-1}\vee\sigma(S_n)$ and $S_n\perp \mathcal{F}_{T_{n-1}}$ the tower property \cite[Section 9.7.(i)]{williams1991probability} implies, 
\[\E\left[Y_{T_{n-1}}|\mathcal{G}_n\right]=\E\left[Y_{T_{n-1}}|\mathcal{G}_{n-1}\right].\]

\noindent\textbf{Step 3.} \textit{Control of the error term}\\

We are now looking for a uniform bound on the error term which depends only on $N$. Recall that,
\[e^{I_n}_{T_n^-} = \left|\frac{1}{N}\sum_{j=1}^N K\left(\overline{X}^{I_n}_{T_{n}^-}-\overline{X}^j_{T_n^-}\right)\overline{m}^{j}_{T_{n}^-}-\iint_{\R^d\times \MM} K\left(\overline{X}^{I_n}_{T_n^-}-x\right)mf_{T_n^-}(x,m)\dd x\dd m\right|.\]
Let us define: 
\[b\left(\overline{Z}^i_{T_n^-},\overline{Z}^j_{T_n^-}\right):=K\left(\overline{X}^{i}_{T_{n}^-}-\overline{X}^j_{T_n^-}\right)\overline{m}^{j}_{T_{n}^-}-\iint_{\R^d\times \MM} K\left(\overline{X}^{i}_{T_n^-}-x\right)mf_{T_n^-}(x,m)\dd x\dd m.\]
These quantities are uniformly bounded:
\[\left|b\left(\overline{Z}^i_{T_n^-},\overline{Z}^j_{T_n^-}\right)\right|\leq C\|K\|_{L^\infty}\]
and satisfy (by pairwise independence of the McKean processes): 
\begin{equation}\label{centeringproperty}\E\left[b\left(\overline{Z}^i_{T_n^-},\overline{Z}^j_{T_n^-}\right)\Big|\overline{Z}^i_{T_n^-}\right]=0.\end{equation}
The Cauchy-Schwarz inequality implies
\begin{align*}
\E\Big[e^{I_n}_{T_n^-}\Big|\mathcal{G}_n\Big]^2&=\E\left[\left|\frac{1}{N}\sum_{j=1}^{N}b\left(\overline{Z}^{I_n}_{T_n^-},\overline{Z}^j_{T_n^-}\right)\right|\Big|\mathcal{G}_n\right]^2\\
&\leq\E\left[\frac{1}{N^2}\left(\sum_{j=1}^{N}b\left(\overline{Z}^{I_n}_{T_n^-},\overline{Z}^j_{T_n^-}\right)\right)^2\Big|\mathcal{G}_n\right]\\
&\hspace{0.5cm}= \frac{1}{N^2}\sum_{j,k=1}^{N}\E\left[b\left(\overline{Z}^{I_n}_{T_n^-},\overline{Z}^j_{T_n^-}\right)b\left(\overline{Z}^{I_n}_{T_n^-},\overline{Z}^k_{T_n^-}\right)\Big|\mathcal{G}_n\right].
\end{align*}
And because of the centering property \eqref{centeringproperty} and the pairwise independence, if $I_n\ne j$ and $I_n\ne k$:
\begin{align*}\E\left[b\left(\overline{Z}^{I_n}_{T_n^-},\overline{Z}^j_{T_n^-}\right)b\left(\overline{Z}^{I_n}_{T_n^-},\overline{Z}^k_{T_n^-}\right)\Big|\mathcal{G}_n\right]&=\E\left[b\left(\overline{Z}^{I_n}_{T_n^-},\overline{Z}^j_{T_n^-}\right)\E\left[b\left(\overline{Z}^{I_n}_{T_n^-},\overline{Z}^k_{T_n^-}\right)\Big|\overline{Z}^{I_n}_{T_n^-},\overline{Z}^j_{T_n^-}\right]\Big|\mathcal{G}_n\right]\\
&=0,\end{align*}
which implies that there are only $\OO(N)$ non-zero terms in the sum. Since these terms are uniformly bounded, we deduce:
\begin{equation}\label{VII}\E\Big[e^{I_n}_{T_n^-}\Big|\mathcal{G}_n\Big]\leq C\sqrt{\frac{\|K\|_{L^\infty}}{N}}.\end{equation}

\begin{rem}
Note that in the second step we use only the bound \eqref{wassersteinboundM} which is weaker than the assumption  \eqref{fluxlipschitz}. The assumption \eqref{fluxlipschitz} is only crucial in this proof to obtain the control of the error term in $1/\sqrt{N}$. However, a similar conclusion could be reached by applying more general quantitative results on the convergence in Wasserstein distance of the empirical measures of i.i.d random variables towards their common law, see for example \cite{fournier2015rate}. This approach would require the adaptation of the existing results to our geometrical framework and we therefore choose to assume \eqref{fluxlipschitz} and use the more straightforward computation presented in step 3 as in \cite{sznitman89} and since it corresponds to the models \cite{toto,dimarcomotsch16}. 
\end{rem}

\noindent\textbf{Step 4.} \textit{Gronwall lemma} \\ 

From the \eqref{VI} and \eqref{VII} we conclude that for all $n\geq1$,
\begin{equation}\label{recurrence}
\E[Y_{T_n}|\mathcal{G}_n]\leq(1+\lambda S_n)\left(1+\frac{\sigma(K)}{N}\right)\E[Y_{T_{n-1}}|\mathcal{G}_{n-1}]+ C\frac{\theta\big(\|K\|_{L^\infty}\big)\sqrt{\|K\|_{L^\infty}}}{\sqrt{N}}\end{equation}
and $Y_0=0$. 
This is a discrete Gronwall-type inequality and using the inequality $1+x\leq e^x$, it can be seen by induction that for all $n\geq1$: 
\begin{equation}\label{gronwall}
\E[Y_{T_n}|\mathcal{G}_n]\leq C\frac{\theta\big(\|K\|_{L^\infty}\big)\sqrt{\|K\|_{L^\infty}}}{\sqrt{N}}\,\sum_{k=0}^{n-1}\e^{\frac{\sigma(K)}{N}k+\left(\lambda+\frac{\sigma(K)}{N}\right)(T_n-T_{n-k})}.
\end{equation}

\noindent\textbf{Step 5.} \textit{Conclusion : bound at time $t>0$}\\

In order to bound $\E[Y_t]$, we first look for a bound on $\E[Y_{T_{N_t}}]$ where
\[N_t = \sup\{n\in\N,\,\,T_n\leq t\}.\]
The random variable $N_t$ has a Poisson distribution of parameter $Nt$ \cite[Chapter 8.1.2]{bremaud} and in particular $\PP(N_t=n)=\frac{(Nt)^n}{n!}\e^{-Nt}$. Since everything is deterministic on $[T_{N_t},t)$, we will then use the fact that: 
\begin{equation}\label{VIII}\E[Y_t]\leq (1+\lambda t) \E[Y_{T_{N_t}}].\end{equation}
Using the relation \eqref{gronwall}, we first have:
\begin{align*}
\E[Y_{T_n}\1_{N_t=n}] &= \E\Big[\E\left[Y_{T_n}\1_{N_t=n}\big|\mathcal{G}_{n+1}\right]\Big]\\
& =  \E\Big[\1_{N_t=n}\E\left[Y_{T_n}\big|\mathcal{G}_{n+1}\right]\Big]\\
&= \E\Big[\1_{N_t=n}\E\left[Y_{T_n}\big|\mathcal{G}_{n}\right]\Big]\\
&\leq C\frac{\theta\big(\|K\|_{L^\infty}\big)\sqrt{\|K\|_{L^\infty}}}{\sqrt{N}}\E\left[\sum_{k=0}^{n-1} \e^{\frac{\sigma(K)}{N}k+\left(\lambda+\frac{\sigma(K)}{N}\right)t}\1_{N_t=n}\right]\\
&\leq C\frac{\theta\big(\|K\|_{L^\infty}\big)\sqrt{\|K\|_{L^\infty}}}{\sqrt{N}}\e^{\left(\lambda+\frac{\sigma(K)}{N}\right)t}\frac{\e^{\frac{\sigma(K)n}{N}}}{\e^{\frac{\sigma(K)}{N}}-1}\PP(N_t=n)
\end{align*}
And therefore,
\begin{align*}
\E\left[Y_{T_{N_t}}\right]&=\sum_{n=0}^\infty \E[Y_{T_n}\1_{N_t=n}]\\
&\leq C\frac{\theta\big(\|K\|_{L^\infty}\big)\sqrt{\|K\|_{L^\infty}}}{\sqrt{N}}\frac{\e^{\left(\lambda+\frac{\sigma(K)}{N}\right)t}}{\e^{\frac{\sigma(K)}{N}}-1}\sum_{n=0}^\infty \e^{\frac{\sigma(K)}{N}n}\PP(N_t=n) \\ 
&\leq C\frac{\theta\big(\|K\|_{L^\infty}\big)\sqrt{\|K\|_{L^\infty}}}{\sqrt{N}}\frac{\e^{\left(\lambda+\frac{\sigma(K)}{N}\right)t}}{\e^{\frac{\sigma(K)}{N}}-1}\exp\left(Nt\left(\e^{\frac{\sigma(K)}{N}}-1\right)\right).
\end{align*}
Finally, from this last expression and expression \eqref{VIII} we conclude that
\begin{multline}\label{IX}\frac{1}{N}\E[Y_t]\leq (1+\lambda t)\frac{1}{N}\E\left[Y_{T_{N_t}}\right]\\
\leq C\frac{\theta\big(\|K\|_{L^\infty}\big)\sqrt{\|K\|_{L^\infty}}}{\sqrt{N}}\frac{\e^{\left(2\lambda+\frac{\sigma(K)}{N}\right)t}}{N\left(\e^{\frac{\sigma(K)}{N}}-1\right)}\exp\left(Nt\left(\e^{\frac{\sigma(K)}{N}}-1\right)\right)\end{multline}
where we have used the fact that $1+\lambda t\leq\e^{\lambda t}$. To prove the first point of Theorem \ref{propagationofchaostheorem} we notice that 
\begin{align*}
\E[W^1(\hat{\mu}^N_t,f_t)] &\leq \E[W^1(\hat{\mu}^N_t,\overline{\mu}^N_t)]+\E[W^1(\overline{\mu}^N_t,f_t)]\\ 
&\leq \frac{1}{N}\E[Y_t]+\E[W^1(\overline{\mu}^N_t,f_t)]
\end{align*}
where the bound between the two empirical measures is a consequence of the Kantorovich dual formulation \eqref{kantorovich} of the Wasserstein distance. Since $\E[W^1(\overline{\mu}^N_t,f_t)]\to0$ as $N\to+\infty$ by independence of the McKean processes, the result follows from \eqref{IX}.\\

The second point of Theorem \ref{propagationofchaostheorem} follows from the definition of the Wasserstein distance which implies that
\[W^1(\mu^N_t,f_t)\leq\frac{1}{N}\E[Y_t].\]
The final bound is obtained from the right-hand side of \eqref{IX} and the fact that for all $x\geq0$, $x\leq \e^{x}-1\leq x\e^x$.  
\end{proof}

\begin{rem}
The above proof is based on the same classical coupling arguments of \cite{sznitman89} but in a PDMP framework. This type of proof is rather elementary and allows an explicit control of the error. In the context of jump processes, another natural approach would be to re-write the IBM as a system of stochastic differential equations with respect to a Poisson measure. The adaptation of the propagation of chaos result in this framework and the topological questions that it may raise are left for future work. 
\end{rem}

\section{Moderate interaction (proof of Theorem \ref{moderateinteractiontheorem})}\label{moderateinteractionsection}

This section is devoted to the proof of Theorem \ref{moderateinteractiontheorem}. We follow the approach of \cite{jourdainmeleard}. Given a sequence of kernels $(K^N)_N$, the goal is to apply Ascoli-Arzel\`a's theorem to show the convergence in $C([0,T],B_a)$ of a subsequence of the sequence of solutions of  \eqref{bgkkernel} associated to $(K^N)_N$ towards the unique solution of \eqref{bgkdelta}. To do so, we will first need to prove some properties on the associate sequence of solutions of \eqref{bgkkernel} (Lemma \ref{lemmapropertiesfN}) and in particular an equicontinuity property. We will then prove a compactness result (Lemma \ref{rfkmanifold}) which is an adaptation of the classical Riesz-Fréchet-Kolmogorov to our geometrical context and which will be needed to prove the compactness of the sequence at a fixed time (pointwise compactness, Lemma \ref{pointwisecompactness}). The proof of Theorem \ref{moderateinteractiontheorem} can be found in Section \ref{proofmoderateinteractionsection}. 

\subsection{First estimates}\label{firstestimatessubsection}

From now, we fix $a>0$ and $T>0$ such that \eqref{bgkdelta} is wellposed on $C([0,T],B_a)$ where $B_a$ is the ball of radius $a>0$ in $L^\infty\cap\widetilde{\mathcal{P}}(\R^d\times\MM)$. A key remark is tha, if $g\in B_a$ then for all $N\in\N$ and all $x\in\R^d$, 
\[|J_{K^N*g}(x)|\leq \iint_{\R^d\times\MM} K^N(x-y)|g(y,m)|\dd y\dd m\leq a,\]
since the integral of $K^N$ is equal to 1 and where we have used that $|m|\leq1$ for all $m\in\mathscr{M}$. As a consequence, the solution $f^N$ of \eqref{bgkkernel} with kernel $K^N$ and initial condition $f_0$ belongs to $C([0,T],B_a)$ as it can be seen by rewriting the proof of Proposition \ref{wellposednessdelta} for the interaction law with kernel $K^N$. \\

\begin{lemma}\label{lemmapropertiesfN} The sequence $(f^N)_N\in C([0,T],B_a)$ satisfies the following properties.
\begin{enumerate}[(i)]
\item (space-translation stability) Let $h\in\R^d$, then 
\begin{equation}\label{spacetranslationstability}\sup_{t\in[0,T]}\|\tau_{(h,id)}f^N_t-f^N_t\|_{L^1(\R^d\times\MM)}\leq C(a,T)\|\tau_{(h,id)}f_0-f_0\|_{L^1(\R^d\times\MM)},\end{equation}
where the translation operator $\tau$ is defined for a measurable function $g$ on $\R^d\times\MM$, $h\in\R^d$ and $\psi\in C(\MM,\MM)$ by $(\tau_{(h,\psi)}g)(x,m):=g(x+h,\psi(m))$.
\item (tightness) There exists $\beta>0$ such that for $R>0$ sufficiently big, then 
\begin{equation}\label{tightness}\sup_{t\in[0,T]}\iint_{\{|x|\geq R\}\times \MM} f^N_t(x,m)\dd x\dd m\leq C(T) \iint_{\{|x|\geq R-\beta T\}\times \MM} f_0(x,m)\dd x\dd m.\end{equation}
\item (equicontinuity) Let $t\in[0,T)$ and $h>0$ sufficiently small, then 
\begin{equation}\label{equicontinuity}\|f^N_{t+h}-f^N_t\|_{L^1(\R^d\times\MM)}\leq  C(a,T)\Big(\|\mathsf{T}_hf_0-f_0\|_{L^1(\R^d\times\MM)}+\varepsilon(h)\Big),\end{equation}
where $\varepsilon(h)\to0$ uniformly in $h$ and $\mathsf{T}_h$ is the free-transport operator defined by \eqref{freetransportoperator}. 
\end{enumerate}
\end{lemma}

\begin{proof} Recall that $f^N$ is given by Duhamel's formula: 
\begin{multline*}f_t^N(x,m)=\e^{-t}f_0(x-t\Phi(m),m)\\
+\int_{0}^t \e^{-(t-s)} \rho_{f^N_s}(x-(t-s)\Phi(m))M_{K^N*f^N_s}[x-(t-s)\Phi(m)](m)\,\dd s.\end{multline*}
Or equivalently: 
\[f_t^N=\e^{-t}\mathsf{T}_tf_0+\int_0^t \e^{-(t-s)}\, \mathsf{T}_{t-s}G^{K}_{f^N_s}\,\dd s,\]
where
\[G^K_f(x,m) :=\rho_f(x)M_{K^N*f}[x](m).\]
\begin{enumerate}[(i)]
\item Since $\tau_{(h,id)}\mathsf{T}_t=\mathsf{T}_t\tau_{(h,id)}$ and $\tau_{(h,id)}G^K_f=G^K_{\tau_{(h,id)}f}$, the conclusion follows from the stability estimate obtained by Duhamel formula by using the Lipschitz bound \eqref{Glipschitzl1} : 
\[\|G^K_f-G^K_g\|_{L^1(\R^d\times\MM)}\leq (\alpha(a)+a\theta(a))\|f-g\|_{L^1(\R^d\times\MM)}\]
and Gronwall's lemma. 
\item For a given $R>0$
\begin{multline*}
\iint_{\{|x|\geq R\}\times \MM} f^N_t(x,m)\dd x\dd m=\e^{-t}\iint_{\{|x|\geq R\}\times \MM} f_0(x-t\Phi(m),m)\dd x \dd m \\
+\int_0^t \e^{-(t-s)} \iint_{\{|x|\geq R\}\times \MM} G^K_{f^N_s}(x-(t-s)\Phi(m),m)\dd x\dd m\,\dd s.
\end{multline*}
We define for $0\leq s\leq t$ and for $R>0$ sufficiently big: 
\[\mathcal{I}^N_R(s):=\iint_{\{|x|\geq R-\beta(t-s)\}\times \MM}f^N_s(x,m)\dd x \dd m\]
where $\beta>0$ is such that $|\Phi(m)|\leq\beta$ for all $m\in \MM$. For fixed $m\in\MM$ and $s\in[0,t]$, with the change of variables $x\mapsto x-(t-s)\Phi(m)$, it holds that: 
\begin{multline*}\int_{|x|\geq R} G^K_{f^N_s}(x-(t-s)\Phi(m),m)\dd x 
= \int_{|x+(t-s)\Phi(m)|\geq R} \rho_{f^N_s}(x)M_{K^N*f^N_s}[x](m)\dd x
\\ \leq \int_{|x|\geq R-\beta(t-s)} \rho_{f^N_s}(x)M_{K^N*f^N_s}[x](m)\dd x,
\end{multline*}
where the last inequality comes from the fact that $\{x,\,\,|x+(t-s)\Phi(m)|\geq R\}\subset\{x,\,\,|x|\geq R-\beta(t-s)\}$. Integrating this expression on $\MM$ and using the fact that $M_{K^N*f^N_s}[x]$ is a probability density function, we obtain: 
\[\iint_{\{|x|\geq R\}\times\MM} G^K_{f^N_s}(x-(t-s)\Phi(m),m)\dd x \leq \mathcal{I}^N_{R}(s).\]
Then, since $e^{-t}\leq 1$, it holds that 
\[\mathcal{I}^N_R(t)\leq \mathcal{I}^N_R(0)+\int_0^t \mathcal{I}^N_R(s)\e^s\dd s\]
from which we can conclude using Gronwall lemma that: 
\[\mathcal{I}^N_R(t)=\iint_{\{|x|\geq R\}\times \MM} f^N_t(x,m)\dd x\dd m\leq C(t)\iint_{\{|x|\geq R-\beta t\}\times \MM} f_0(x,m)\dd x\dd m=\mathcal{I}^N_R(0).\]
\item We have: 
\begin{multline*}\|f^N_{t+h}-f^N_t\|_{L^1(\R^d\times\MM)}\leq\e^{-t}\|\mathsf{T}_hf_0-f_0\|_{L^1(\R^d\times\MM)}\\
+\int_0^t\e^{-(t-s)}\|\mathsf{T}_{h}G^K_{f^N_s}-G^K_{f^N_s}\|_{L^1(\R^d\times\MM)}+\varepsilon(|h|)\end{multline*}
where $\varepsilon(|h|)\to0$ and depends only on $a$ and $T$ (it comes from the terms $\e^{-(t+h)}$ instead of $\e^{-t}$ and the integral between $t$ and $t+h$). Then it holds that
\begin{align*}
&\|\mathsf{T}_{h}G^K_{f^N_s}-G^K_{f^N_s}\|_{L^1(\R^d\times\mathscr{M})}\\
&\hspace{1cm}=\iint_{\R^d\times \MM} \Big|\int_{\MM} f^N_s\left(x-h\Phi(m),m'\right)M_{K^N*f^N_s}\left[x-h\Phi(m)\right](m)\\
&\hspace{5cm}-f^N_s(x,m')M_{K^N*f^N_s}\left[x\right](m)\dd m'\Big|\dd x\dd m\\
&\hspace{1cm}\leq\int_{\MM}\Big\{ \iint_{\R^d\times \MM}  \Big|f^N_s\left(x-h\Phi(m),m'\right)M_{K^N*f^N_s}\left[x-h\Phi(m)\right](m)\\
&\hspace{5cm}-f^N_s(x,m')M_{K^N*f^N_s}\left[x\right](m)\Big|\dd m'\dd x\Big\}\dd m\\
\end{align*}
And it holds that for a given $m\in \MM$: 
\begin{align*}
&\iint_{\R^d\times \MM}  \Big|f^N_s\left(x-h\Phi(m),m'\right)M_{K^N*f^N_s}\left[x-h\Phi(m)\right](m)\\
&\hspace{7cm}-f^N_s(x,m')M_{K^N*f^N_s}\left[x\right](m)\Big|\dd m'\dd x\\
&\hspace{0.2cm}\leq
\alpha(a)\left\|(\tau_{(-h\Phi(m),id)}f^N_s)-f^N_s\right\|_{L^1(\R^d\times\MM)}\\
&\hspace{1cm}+Ca\theta(a)\iint_{\R^d\times\MM}\iint_{\R^d\times\MM} K^N(y)\big|f^N_s(x-h\Phi(m)-y,m'')\\
&\hspace{7.5cm}-f^N_s(x-y,m'')\big||m''|\dd y\dd m''\dd x\dd m'.
\end{align*}
Since $|m''|\leq1$ and $\int_\mathscr{M}\dd m'=1$, by switching the integrals in $x$ and in $y$ (Fubini's theorem), with the change of variable $x'=x-y$ and since the integral of $K^N$ is equal to 1, it holds that: 
 \begin{multline*}
 \iint_{\R^d\times\MM}\iint_{\R^d\times\MM} K^N(y)\left|f^N_s(x-h\Phi(m)-y,m'')-f^N_s(x-y,m'')\right||m''|\dd y\dd m''\dd x\dd m'\\
 \leq \left\|\tau_{(-h\Phi(m),id)}f^N_s-f^N_s\right\|_{L^1(\R^d\times\MM)}.
 \end{multline*}
Gathering everything and using the space-translation stability property, we obtain: 
\[\|\mathsf{T}_{h}G^K_{f^N_s}-G^K_{f^N_s}\|_{L^1(\R^d\times\mathscr{M})}\leq c(a,T)\int_{\mathscr{M}} \|\tau_{(-h\Phi(m),id)}f_0-f_0\|_{L^1(\R^d\times\mathscr{M})}\dd m,\]
where $c(a,T)>0$ is a constant which depends only on $a$, $T>0$ and $\MM$. This last term vanishes uniformly on $h$ by the dominated convergence theorem and the continuity of the shift operator in $L^1$ (as it can be seen by approximation by smooth functions \cite[p.79]{aubin}).
\end{enumerate}
\end{proof}

\begin{rem}
These properties can be understood at the level of the processes $\overline{Z}_t$. In particular, we can deduce some of them from \cite{monmarche}: 
\begin{enumerate}[$\bullet$]
\item Since $\tau_{(h,id)}$ commutes with everything, the space-translation stability follows from the stability result \cite[Theorem 3]{monmarche}.
\item For the equicontinuity, starting from $f_0$, let's run the process up to time $h$. Since everything is Markovian, we can consider the two nonlinear processes starting from $f_0$ and $f_h$. Their laws at time $t$ are respectively $f_t$ and $f_{t+h}$. The stability result \cite[Theorem 3]{monmarche} shows that: 
\[\|f_{t+h}-f_t\|_{TV}\leq C(a,T)\|f_h-f_0\|_{TV}.\]
To conclude, it is enough to write: 
\[\|f_h-f_0\|_{TV}\leq \|\mathsf{T}_hf_0-f_0\|_{TV}+\|f_h-\mathsf{T}_hf_0\|_{TV}\]
and to notice that:
\[\|f_h-\mathsf{T}_h f_0\|_{TV}\leq \PP(\text{there is a jump on }[0,h])=\OO(h).\]
\end{enumerate}
\end{rem}

\subsection{A compactness result}

We first state a compactness lemma, which is the analog of the Riesz-Fréchet-Kolmogorov theorem in a Riemannian setting. The proof can be found in Appendix \ref{proofRFK}. The ideas of the proof come from \cite[Chapter 2]{hebey2000nonlinear}. When $\omega$ and $\Omega$ are two open sets in a metric space, we write $\omega\subset\subset\Omega$ when $\omega\subset\Omega$, $\overline{\omega}\subset\Omega$ and $\overline{\omega}$ is compact where $\overline{\omega}$ is the closure of $\omega$ in the ambiant metric space. 
\begin{lemma}\label{rfkmanifold}
Let $\mathcal{F}\subset L^1(\R^d\times\MM)$ be a bounded subset which satisfies the following properties.
\begin{enumerate}
\item For all $\varepsilon>0$, there exists $R>0$ such that:
\[\forall f\in\mathcal{F},\,\,\,\iint_{\{|x|\geq R\}\times\MM} f(x,m)\dd x\dd m <\varepsilon.\]
\item For all $\varepsilon>0$ and for all $\omega\subset\subset\Omega\subset\MM$ open sets, there exists $\eta >0$ such that for any smooth function $\phi : \omega\to\Omega$ which satisfies
\[\forall m\in\omega,\,\,\,d(\phi(m),m)<\eta,\]
and for all $h\in\R^d$ such that $|h|<\eta$, it holds that:
\[\forall f\in\mathcal{F},\,\,\,\iint_{\R^d\times\omega}|f(x+h,\phi(m))-f(x,m)|\dd x\dd m <\varepsilon.\]
\end{enumerate}
Then $\mathcal{F}$ is relatively sequentially compact in $L^1(\R^d\times\MM)$. 
\end{lemma}

Using this lemma, it can be shown that for every $t\in[0,T]$, the sequence $(f_t^N)_N$ is compact in $L^1(\R^d\times \MM)$.

\begin{lemma}[Pointwise compactness]\label{pointwisecompactness} Let $t\in[0,T]$ be a given time. Then the sequence $(f^N_t)_N$ of solutions of \eqref{bgkkernel} associated to the sequence of kernels $(K^N)_N$ is compact in $L^1(\R^d\times\MM)$. 
\end{lemma}

\begin{proof} To prove this result, we will show that the family $(f^N_t)_N$ satisfies the conditions to apply Lemma \ref{rfkmanifold}. Firstly we see that the family is bounded in $L^1$ : for all $N\in\N$, $\|f^N_t\|_{L^1(\R^d\times\mathscr{M})}=1$. Secondly thanks to the tightness result \eqref{tightness} the family satisfies the first point of Lemma \ref{rfkmanifold}. Finally we are left with proving that the family satisfies the second point in Lemma \ref{rfkmanifold}. The conclusion is then a direct consequence of Lemma \ref{rfkmanifold}. The rest of this proof is devoted to showing that the family $(f^N_t)_N$ satisfies condition 2 of Lemma \eqref{rfkmanifold}. \\ 

Let $h\in\R^d$, $\omega\subset\subset\Omega\subset\MM$ open sets and $\psi : \omega\to\Omega$ a continuous function. Recall the Duhamel's formula: 
\[f_t^N=\e^{-t}\mathsf{T}_tf_0+\int_0^t \e^{-(t-s)}\, \mathsf{T}_{t-s}G^K_{f^N_s}\,\dd s\]
where the free-transport operator $\mathsf{T}_t$ is defined by \eqref{freetransportoperator} and the $(h,\psi)$-translation is defined by:
\[\tau_{(h,\psi)} f(x,m):=f(x+h,\psi(m)).\]
To show that the family $(f^N_t)_N$ fulfills condition 2 in Lemma \ref{rfkmanifold}, we have to bound:
\begin{align}\label{fourterms}
\nonumber\|\tau_{(h,\psi)}f^N_t-f^N_t\|_{L^1(\R^d\times\omega)}\leq\, &\e^{-t}\|\tau_{(h,\psi)}\mathsf{T}_t f_0-\mathsf{T}_t f_0\|_{L^1(\R^d\times\omega)}\\
\nonumber&+\int_0^t\e^{-(t-s)} \|\tau_{(h,\psi)}\mathsf{T}_{t-s} G^K_{f^N_s}-\mathsf{T}_{t-s} G^K_{f^N_s}\|_{L^1(\R^d\times\omega)}\,\dd s\\[0.1cm]
\nonumber\leq\,& \e^{-t}\|\tau_{(h,\psi)}f_0-f_0\|_{L^1(\R^d\times\omega)}\\[0.1cm]
\nonumber&+\e^{-t}\|\tau_{(h,\psi)}\mathsf{T}_tf_0-\mathsf{T}_t\tau_{(h,\psi)}f_0\|_{L^1(\R^d\times\omega)}\\[0.1cm]
\nonumber&+\int_0^t\e^{-(t-s)} \|\tau_{(h,\psi)}G^K_{f^N_s}-G^K_{f^N_s}\|_{L^1(\R^d\times\omega)}\,\dd s\\[0.1cm]
\nonumber&+\int_0^t\e^{-(t-s)}\|\tau_{(h,\psi)}\mathsf{T}_{t-s} G^K_{f^N_s}-\mathsf{T}_{t-s} \tau_{(h,\psi)}G^K_{f^N_s}\|_{L^1(\R^d\times\omega)}\,\dd s\\[0.1cm]
&=:I_1+I_2+I_3+I_4
\end{align}
where we have used that the free-transport operator is an $L^1$-isometry. We now bound each of the four terms in the right hand side.
\begin{enumerate}[$\bullet$]
\item (\textbf{Term} $I_1$) For smooth compactly supported $f_0$, it follows by the dominated convergence theorem that: 
\[\|\tau_{(h,\psi)}f_0-f_0\|_{L^1(\R^d\times\omega)}\to0\]
when $h\to0$ and $\psi\to id$ (for the uniform convergence topology on $\mathscr{M}$). This limit still holds for any $f_0\in L^1(\R^d\times \omega)$ by density in the $L^1$ norm of smooth compactly supported functions and the fact that the translation operator is an isometry. 
\item (\textbf{Term} $I_2$) By similar arguments, 
\begin{align*}
&\|\tau_{(h,\psi)}\mathsf{T}_tf_0-\mathsf{T}_t\tau_{(h,\psi)}f_0\|_{L^1(\R^d\times\omega)}\\
&\hspace{1cm}= \iint_{\R^d\times\omega} \left|\mathsf{T}_t f_0(x+h,\psi(m))-(\tau_{(h,\psi)}f_0)(x-t\Phi(m),m)\right|\dd x\dd m \\ 
&\hspace{1cm}= \iint_{\R^d\times\omega} \left|f_0\left(x+h-t\Phi\big(\psi(m)\big),\psi(m)\right)-f_0\left(x+h-t\Phi(m),\psi(m)\right)\right|\dd x\dd m \\ 
&\hspace{1cm}= \iint_{\R^d\times\omega} \left|f_0\left(x-t\left(\Phi\big(\psi(m)-\Phi(m)\right),\psi(m)\right)-f_0(x,\psi(m))\right|\dd x\dd m
\end{align*}
converges to 0 as $h\to0$ and $\psi\to id$ since $\Phi$ is Lipschitz. 
\item (\textbf{Term} $I_3$) We write the third term as: 
\begin{align*}
&\|\tau_{(h,\psi)}G^K_{f^N_s}-G^K_{f^N_s}\|_{L^1(\R^d\times\omega)} \\[0.2cm]
&=\iint_{\R^d\times\omega}\Big|\int_\MM f^N_s(x+h,m') M_{K^N*f^N_s}[x+h]\big(\psi(m)\big)\\
&\hspace{7cm}-f^N_s(x,m')M_{K^N*f^N_s}[x](m)\dd m'\Big|\dd x\dd m\\[0.2cm]
&\leq\int_\MM \iint_{\R^d\times\omega}|f^N_s(x+h,m')-f^N_s(x,m')|M_{K^N*f^N_s}[x+h]\big(\psi(m)\big)\dd x\dd m\dd m'\\[0.2cm]
&\hspace{1cm}+\int_\MM\iint_{\R^d\times\omega}f^N_s(x,m')\left|M_{K^N*f^N_s}[x+h]\big(\psi(m)\big)-M_{K^N*f^N_s}[x](m)\right|\dd x\dd m\dd m'\\[0.2cm]
&\leq \alpha(a)\|\tau_{(h,id)}f^N_s-f^N_s\|_{L^1(\R^d\times\MM)} \\[0.2cm]
&\hspace{0.5cm}+\int_\MM\iint_{\R^d\times\omega}f^N_s(x,m')\left|M_{K^N*f^N_s}[x+h]\big(\psi(m)\big)-M_{K^N*f^N_s}[x]\big(\psi(m)\big)\right|\dd x\dd m\dd m'\\[0.2cm]
&\hspace{1cm}+\int_\MM\iint_{\R^d\times\omega}f^N_s(x,m')\left|M_{K^N*f^N_s}[x]\big(\psi(m)\big)-M_{K^N*f^N_s}[x](m)\right|\dd x\dd m\dd m'\\[0.2cm]
&\leq \alpha(a)\|\tau_{(h,id)}f^N_s-f^N_s\|_{L^1(\R^d\times\MM)} \\[0.2cm]
&\hspace{1cm}+a\int_{\R^d}\left\|M_{K^N*f^N_s}[x+h]-M_{K^N*f^N_s}[x]\right\|_{L^\infty(\mathscr{M})}\dd x\\[0.2cm]
&\hspace{1.5cm}+L(a)\sup_{m\in\omega}d(\psi(m),m)
\end{align*}
where we have used the fact that for all $N\in\N$ and all $t\in[0,T]$, $\|f^N_t\|_{L^\infty(\R^d\times\mathscr{M})}\leq a$ where $a>0$ is fixed (see beginning of Subsection \ref{firstestimatessubsection}). Using the bound \eqref{fluxlipschitz}, Fubini's theorem and the fact that the integral of $K^N$ is equal to 1, it holds that
\begin{align*}
&\int_{\R^d}\left\|M_{K^N*f^N_s}[x+h]-M_{K^N*f^N_s}[x]\right\|_{L^\infty(\mathscr{M})}\dd x\\
&\hspace{2cm}\leq\theta(a)\int_{\R^d}\iint_{\R^d\times\MM} K^N(y)|f^N_s(x+h-y,m)-f^N_s(x-y,m)|\dd y\dd m\dd x \\
&\hspace{2.5cm}=\theta(a)\|\tau_{(h,id)}f^N_s-f^N_s\|_{L^1(\R^d\times\MM)}
\end{align*}
Finally we conclude the following estimate:
\begin{multline*}\|\tau_{(h,\psi)}G^K_{f^N_s}-G^K_{f^N_s}\|_{L^1(\R^d\times\omega)}\leq (\alpha(a)+a\theta(a))\|\tau_{(h,id)}f^N_s-f^N_s\|_{L^1(\R^d\times\MM)}\\
+L(a)\sup_{m\in\omega}d(\psi(m),m).\end{multline*}
Using the space-translation stability \eqref{spacetranslationstability} we deduce that this term goes to zero when $h\to0$ and $\psi\to id$ (for the uniform convergence topology on $\mathscr{M}$). 
\item (\textbf{Term} $I_4$) We bound the fourth term: 
\begin{align*}
&\|\tau_{(h,\psi)}\mathsf{T}_{t-s} G^K_{f^N_s}-\mathsf{T}_{t-s} \tau_{(h,\psi)}G^K_{f^N_s}\|_{L^1(\R^d\times\omega)}\\
&\hspace{1cm}=\iint_{\R^d\times\omega} \left|G^K_{f^N_s}\left(x-(t-s)\left(\Phi\big(\psi(m)\big)-\Phi(m)\right),\psi(m)\right)-G^K_{f^N_s}(x,\psi(m))\right|\dd x\dd m\\[0.1cm]
&\hspace{1cm}\leq\iint_{\R^d\times\omega}\int_\MM \Big|f^N_s\left(x-(t-s)\left(\Phi\big(\psi(m)\big)-\Phi(m)\right),m'\right)\\
&\hspace{4cm}\times M_{K^N*f^N_s}[x-(t-s)\left(\Phi\left(\psi(m)\right)-\Phi(m)\right)]\big(\psi(m)\big) \\
&\hspace{6cm}-f^N_s(x,m')M_{K^N*f^N_s}[x]\big(\psi(m)\big)\Big|\dd m'\dd x\dd m\\
&\hspace{1cm}\leq\big(\alpha(a)+a\theta(a)\big)\int_\omega\left\|\tau_{(-(t-s)\left(\Phi(\psi(m))-\Phi(m)\right),id)}f^N_s-f^N_s\right\|_{L^1(\R^d\times\MM)}\dd m
\end{align*}
Using the space-translation stability and the fact that $\Phi$ is Lipschitz, we conclude that this term goes to zero when $h\to0$ and $\psi\to id$. 
\end{enumerate}
Finally, gathering the analysis of these four terms, we conclude that the right-hand side in expression \eqref{fourterms} converges to 0 when $h\to0$ and $\psi\to id$. Therefore the family $(f^N_t)_N$ fulfills condition 2 in Lemma \ref{rfkmanifold}, as we wanted to show. 
\end{proof}

\begin{rem}
This result is similar to the compact injection $W^{1,1}(M)\hookrightarrow L^1(M)$ where $M$ is compact Riemannian manifold (this follows from Rellich-Kondrakov injection theorem \cite[Theorem 2.9]{hebey2000nonlinear} which is a consequence of Riesz-Fréchet-Kolmogorov theorem, as Lemma \ref{rfkmanifold}). We can apply directly this result under stronger assumptions on $f_0$ : if $f_0$ is compactly supported then it can be seen as a function on $\T^d\times\MM$ where $\T^d$ is a sufficiently big $d$-dimensional torus. The tightness result \eqref{tightness} proves that $f^N_t$ can also be seen as a function on $\T^d\times\MM$. The space-translation stability \eqref{spacetranslationstability} and the Lipschitz bound \eqref{locallylipschitz} on the interaction law ensure that if $f_0\in W^{1,1}(\T^d\times\MM)$ then $(f^N_t)_N$ is bounded in $W^{1,1}(\T^d\times\MM)$ hence compact in $L^1(\T^d\times\MM)$.
\end{rem}

\subsection{Proof of Theorem \ref{moderateinteractiontheorem}}\label{proofmoderateinteractionsection}

\begin{proof}[\sc Proof (of Theorem \ref{moderateinteractiontheorem})]
Under hypothesis \eqref{boundepsilonN}, it holds that
\[\frac{\sigma(K^N)}{N}\underset{N\to+\infty}{\longrightarrow}0\]
and the right-hand side of \eqref{propagationofchaos} (at time $T$) goes to $0$ as $N\to+\infty$ : 
\[\frac{\varepsilon_N^{d/2+1}\e^{T\frac{\sigma(K^N)}{N}}}{\sqrt{N}}\exp\left(T\theta(\varepsilon_N^{-d})\varepsilon_N^{-(d+1)}\e^{\frac{\sigma(K^N)}{N}}\right)\to0,\]
where we can assume without loss of generality that $\|K\|_{L^\infty}=\|K\|_{\mathrm{Lip}}=1$ which implies
\[\|K^N\|_{L^\infty}=\varepsilon_N^{-d}\,\,\,\text{and}\,\,\,\|K^N\|_{\mathrm{Lip}}=\varepsilon^{-(d+1)}.\]
The propagation of chaos result (Theorem \ref{propagationofchaostheorem}) therefore ensures that under Hypothesis \eqref{boundepsilonN},
\[W^{1}(\mu^N_t,f^N_t)\underset{N\to+\infty}{\longrightarrow}0,\]
where $f^N_t$ is the solution of \eqref{bgkkernel} with kernel $K^N$. The Wasserstein-1 convergence implies the weak convergence of measures. It then remains to prove that 
\[\|f^N_t-f_t\|_{L^1(\R^d\times\MM)}\to0,\]
since it also implies the weak convergence of $f^N_t$ towards $f_t$ as measures. The equicontinuity property \eqref{equicontinuity} and Lemma \ref{pointwisecompactness} are the two hypotheses of Ascoli's theorem \cite[Theorem 4.25]{brezis2010functional}. Therefore, we can extract a subsequence of $(f^N)_N$ which converges in $C([0,T],L^1(\R^d\times\MM))$ towards $f\in~C([0,T],L^1(\R^d\times\MM))$. Note that since the sequence is also bounded in $L^\infty(\R^d\times\MM)$, we also have $f_t\in B_a$ for every $t\in[0,T]$. It remains to prove that the limit $f_t$ is uniquely defined as the solution of \eqref{bgkdelta}. Thanks to Duhamel's formula. to do so, it is enough to prove that 
\[\|G^K_{f^N_t}-G_{f_t}\|_{L^1(\R^d\times\MM)}\to0\]
uniformly in $t$. This follows from: 
\begin{align*}
\|G^K_{f^N_t}-G_{f_t}\|_{L^1(\R^d\times\MM}&\leq\alpha(a)\|f^N_t-f_t\|_{L^1(\R^d\times\MM)}\\
&\hspace{1cm}+a\int_{\R^d}\left\|M_{K^N*f^N_t}[x]-M_{K^N*f_t}[x]\right\|_{L^\infty(\MM)}\dd x\\
&\hspace{2cm}+a\int_{\R^d}\left\|M_{K^N*f_t}[x]-M_{f_t}[x]\right\|_{L^\infty(\MM)}\dd x \\ 
&\leq\big(\alpha(a)+a\theta(a)\big)\sup_{t\in[0,T]} \|f^N_t-f_t\|_{L^1(\R^d\times\MM)}\\
&\hspace{1cm}+a\theta(a)\int_{\MM}\|K^N*f_t(\cdot,m)-f_t(\cdot,m)\|_{L^1(\R^d)}\dd m
\end{align*}
and the last term in the right hand side goes to zero by the dominated convergence theorem, uniformly in $t$ by Ascoli's theorem.
\end{proof}

\begin{rem}
If $f_0$ is compactly supported, then 
\[W^1(f^N_t,f_t)\leq C\|f^N_t-f_t\|_{L^1(\R^d\times\MM)}\]
and the convergence of $\mu^N_t$ towards $f_t$ actually holds in Wasserstein-1 distance. 
\end{rem}

\section{An alternative approach in the spatially homogeneous case}\label{alternativesection}

This section is an adaptation of the proof of the mean-field limit of the isotropic 4-wave kinetic equation with bounded jump kernel which can be found in \cite[Section 4.3]{sara}. \\ 

The aim of this section is to write the solution of the space homogeneous BGK equation:
\[\partial_t\nu_t=M_{\nu_t}-\nu_t=:Q_{BGK}(\nu_t),\]
as the mean-field limit of a system of interacting particles. Without loss of generality we will look for a solution is in $\mathcal{P}(\MM)$ (and not only $\mathcal{M}_+(\MM)$) since the total mass is preserved by the equation. In the following, we will say that a $\mathcal{P}(\MM)$-valued process $(\nu_t)_t$ is a weak solution of the space homogeneous BGK equation when it is continuous almost everywhere in time and satisfies for almost every $t\in\R_+$ : 
\[\forall \varphi\in C_b(\MM),\,\,\,\,\,\langle \nu_t,\varphi\rangle= \langle \nu_0,\varphi\rangle + \int_0^t Q_{BGK}^*\varphi(\nu_s)\,\dd s,\]
where for $\varphi\in C_b(\MM)$ and $\nu\in\mathcal{P}(\MM)$, 
\begin{equation}\label{QBGK*}Q_{BGK}^*\varphi(\nu):=\int_{\MM}\int_{\MM} \{\varphi(m')-\varphi(m)\}M_{\nu}(m')\dd m'\dd\nu(m).\end{equation}
Notice that: 
\[Q_{BGK}^*\varphi(\nu)=\langle Q_{BGK}(\nu),\varphi\rangle.\]
The proof of the well-posedness of such weak formulation can be found in \cite[Theorem 6.1]{kolokoltsov_2010}. In the following a \textit{test function} is a continuous bounded function. Notice that since $\MM$ is compact, a continuous function on $\MM$ is automatically bounded. 
\subsection{Individual Based Model}
Consider the $N$ particles $(m^{1,N}_t,\ldots,m^{N,N}_t)\in \MM^N$ dynamics defined by the following jump process:
\begin{enumerate}
\item Let $(T_n)_n$ be an increasing sequence of jump times such that the increments are independent and follow an exponential law of parameter $N$. 
\item At each jump time $T_n$, choose a particle $i\in\{1,\ldots N\}$ uniformly among the $N$ particles and draw the new body-orientation $m^{i,N}_{T_n^+}$ after the jump according to the law $M_{\hat{\mu}^N_{T_n^-}}$ where the empirical measure $\hat{\mu}^N_t$ is defined by
\[\hat{\mu}^N_t:=\frac{1}{N}\sum_{i=1}^N \delta_{m^{i,N}_t}.\]
\end{enumerate}
This defines a continuous time Markov process $(m^{1,N}_t,\ldots,m^{N,N}_t)_t\in\MM^N$ with generator:
\begin{multline*}\mathcal{A}^N \varphi(m_1,\ldots,m_N)\\
:=\sum_{i=1}^{N}\int_{\MM} \big(\varphi(m_1,\ldots,m_i',\ldots,m_N)-\varphi(m_1,\ldots,m_i,\ldots,m_N)\big)M_{\mu^N(m)}(m_i')\,\dd m_i',\end{multline*}
where
\[\mu^N(m)=\frac{1}{N}\sum_{i=1}^{N} \delta_{m_i}.\]

\subsection{A process on measures}

Equivalently, we define a $\mathcal{M}^N(\MM)\subset\mathcal{P}(\MM)$-valued Markov process with generator:
\[G^N\phi(\nu):=N\int_{\MM}\int_{\MM} \left(\phi\left(\nu-\frac{1}{N}\delta_m+\frac{1}{N}\delta_{m'}\right)-\phi(\nu)\right) M_{\nu}(m')\dd m'\dd\nu(m),\]
where $\phi : \mathcal{P}(\MM)\to\R$ is a test function and 
\[\mathcal{M}^N(\MM):=\left\{\frac{1}{N}\sum_{i=1}^N \delta_{m_i},\,\,(m_1,\ldots,m_N)\in\MM^N\right\}\subset\mathcal{P}(\MM)\]
is the set of empirical measures of size $N$.
\begin{lemma}[Links between the IBM and the measure-valued process]\label{IBMmeasureproces} Let $(m^{i,N}_0)_i\in \MM^N$ an initial state and 
\[\hat{\mu}^N_0:=\frac{1}{N}\sum_{i=1}^{N} \delta_{m^{i,N}_0}.\]
It holds that: 
\begin{enumerate}[(i)]
\item the law of any of the $m^{i,N}_t$ is equal to $\E_{\hat{\mu}^{N}_0}\left[\hat{\mu}^{N}_t\right]\in\mathcal{P}(\MM)$,
\item for all $t\in\R_+$, 
\[\E_{\hat{\mu}^{N}_0}\left[\hat{\mu}^{N}_t\right]=\E_{\hat{\mu}^{N}_0}\left[\nu^N_t\right],\]
where $\nu^N_t$ is the $\mathcal{M}^N(\MM)$-valued Markov process with generator $G^N$ and initial state $\hat{\mu}^N_0$.
\end{enumerate}
\end{lemma}

\begin{proof} The first point is proved for example in \cite[Section 1.4]{villani}. To prove the second point let $\varphi : \MM\to\R$ be a test function and 
\[\phi : \mathcal{P}(\MM)\to\R,\,\,\,\nu\mapsto\langle \nu,\varphi\rangle.\]
A direct calculation shows that: 
\[\lim_{t\to0} \frac{\E_{\hat{\mu}^{N}_0}\left[\phi(\widehat{\mu}^N_t)\right]-\phi(\hat{\mu}^N_0)}{t}=G^N\phi(\hat{\mu}^{N}_0).\] 
This proves that $\E_{\hat{\mu}^{N}_0}\left[\phi(\widehat{\mu}^N_t)\right]$ and $\E_{\hat{\mu}^{N}_0}\left[\phi(\nu^N_t)\right]$ satisfy the Kolmogorov equations with same initial condition. In conclusion, for all test function $\varphi$ on $\MM$:
\[\left\langle \E_{\hat{\mu}^{N}_0}\left[\hat{\mu}^N_t\right],\varphi\right\rangle =\left\langle \E_{\hat{\mu}^{N}_0}\left[\nu^N_t\right],\varphi\right\rangle.\]
\end{proof}

\begin{proposition}
Let $\phi : \mathcal{P}(\MM)\to\R$ be a test function and $\nu^N_t$ the Markov process with generator $G^N$. It holds that 
\begin{equation}\label{representationformula}M^{N,\phi}_t:=\phi(\nu^N_t)-\phi(\nu^N_0)-\int_0^t G^N\phi(\nu^N_s)\,\dd s\end{equation}
and
\begin{multline}\label{quadraticvariation}
N_t^{N,\phi}:=\left(M^{N,\phi}_t\right)^2 -N \int_0^t \iint_{\MM\times\MM}\left\{\phi\left(\nu^N_s-\frac{1}{N}\delta_m+\frac{1}{N}\delta_{m'}\right)-\phi(\nu^N_s)\right\}^2\\\times M_{\nu^N_s}(m')\dd m'\dd\nu^N_s(m).\end{multline}
are two martingales.
\end{proposition}

\begin{proof}
These two results are a direct consequence of \cite[Lemma 5.1]{kipnislandim}.
\end{proof}

\subsection{Mean-field limit}

\begin{theorem}\label{meanfieldGN}
Let $\nu_0^N\in \mathcal{M}^N(\MM)$ be an initial state for the process $(\nu^N_t)_t$ with generator $G^N$.
Assume that, as $N\to+\infty$, 
\[\forall\varphi\in C_b(\MM),\,\,\,\langle\nu^N_0,\varphi\rangle \longrightarrow \langle\nu_0,\varphi\rangle.\]
Then, as $N\to+\infty$, the sequence $(\nu^N_t)_t$ converges weakly in $\mathcal{D}\big([0,\infty),\mathcal{P}(\MM)\big)$ towards $(\nu_t)_t$ the deterministic weak solution of the space homogeneous BGK equation with initial condition $\nu_0$, where $\mathcal{D}\big([0,\infty),\mathcal{P}(\MM)\big)$ is the Skorohod space of càdlàg functions from $[0,\infty)$ to $\mathcal{P}(\MM)$.
\end{theorem}

\begin{proof} The proof is split in several steps. \\ 

\noindent\textbf{Step 1.} \textit{The sequence of the laws of the processes $(\nu^N_t)_t$ is tight in $\mathcal{P}\big(\mathcal{D}([0,\infty),\mathcal{P}(\MM))\big)$.} \\ 

Thanks to Jakubowski's criterion \cite[Theorem 4.6]{jakubowski}, it is enough to prove that for all test functions $\varphi:\MM\to\R$, the sequence of the laws of the processes $\big(\langle \nu^N_t,\varphi\rangle\big)_t$ is tight in $\mathcal{P}\big(\mathcal{D}([0,\infty),\R)\big)$ (see \cite[Lemma 4.23]{sara}). This last point is a consequence of the following martingale estimate, for $s\leq t$: 
\[\E\Big[\sup_{s\leq r\leq t}|M^{N,\phi}_r-M^{N,\phi}_s|^2\Big]\leq \frac{4\|\varphi\|^2_\infty}{N}(t-s),\]
where
\[\phi : \mathcal{P}(\MM)\to\R,\,\,\,\nu\mapsto \langle \nu,\varphi\rangle.\]
The tightness of the laws of the processes $\big(\langle \nu^N_t,\varphi\rangle\big)_t$ then follows from the martingale estimates and Aldous criterion \cite[Theorem VI.4.5]{jacodshiryaev}.\\

Finally, by Prokhorov theorem, there exists $(\nu_t)_t\in\mathcal{D}([0,\infty),\mathcal{P}(\MM))$ and we can extract a subsequence, still denoted by $(\nu^N_t)_t$, such that
\[(\nu^N_t)_t\,\Longrightarrow (\nu_t)_t\,\,\,\,\,\,\,\text{as}\,\,\,N\to\infty.\]

\noindent\textbf{Step 2.} \textit{The weak limit $(\nu_t)_t$ is continuous in time a.e. and the convergence is uniform in time.}\\

For any $\varphi \in C_b(\MM)$, we have for almost every $t\in\R_+$:
\[|\langle \nu^N_t,\varphi\rangle - \langle \nu^N_{t^-},\varphi\rangle|\leq \frac{2\|\varphi\|_{L^\infty}}{N},\]
since there is almost surely at most one jump at a given time $t\in\R_+$. 
This proves that $(\nu_t)_t$ is continuous in time a.e.. As a consequence, we obtain by the continuity mapping theorem in the Skorohod space (see \cite[Lemma 4.26]{sara}), that for almost every $t\in\R_+$:
\[\sup_{s\leq t}| \langle\nu^N_s-\nu_s,\varphi\rangle|\to0.\]

\noindent\textbf{Step 3.} \textit{Passing to the limit.}\\

Let $\varphi\in C_b(\MM)$ and $\phi : \mathcal{P}(\MM)\to\R$ defined by:
\[\phi(\nu):=\langle \nu,\varphi\rangle.\]
We want to pass to the limit in the representation formula \eqref{representationformula}.
\begin{enumerate}[$\bullet$]
\item We have made the assumption that 
\begin{equation}\label{convnu0}\langle \nu^N_0,\varphi\rangle \to \langle \nu_0,\varphi\rangle.\end{equation}
\item The martingale estimate \eqref{quadraticvariation} and Doob's inequality ensure that 
\begin{equation}\label{convm}\E\left[\sup_{s\leq t} |M^{N,\phi}_s|^2\right]\leq 4 \E[(M^{N,\phi}_t)^2]\leq \frac{16\|\varphi\|^2_\infty t}{N}\to0.\end{equation}
\item We have for all $s\leq t$: 
\[\iint_{\MM\times\MM} \{\varphi(m')-\varphi(m)\}M_{\nu^N_s}(m')\dd m'\dd\nu^N_s(m)=\int_{\MM} \varphi(m')M_{\nu^N_s}(m')dm' - \langle \nu^N_s,\varphi\rangle.\]
For the second term in the right-hand side we have proved in step 2 that:
\[\sup_{s\leq t}|\langle \nu^N_s-\nu_s,\varphi\rangle|\to0.\]
For the first term in the right-hand side, we use the compactness of $\MM$ and the flux-Lipschitz assumption \eqref{fluxlipschitz} to bound: 
\[\left|\int_{\MM} \varphi(m')M_{\nu^N_s}(m')\dd m'-\int_{\MM} \varphi(m')M_{\nu_s}(m')\dd m'\right|\leq C\|\varphi\|_{L^\infty} |J_{\nu^N_s}-J_{\nu_s}|,\]
where $C>0$ is a constant. Using the uniform convergence proved in step 2 with the coordinates functions in the Euclidean space $E$ (which are continuous and bounded), we deduce that: 
\[\sup_{s\leq t} |J_{\nu^N_s}-J_{\nu_s}| \to 0.\]
Finally, it holds that:
\begin{equation}\label{convG}\int_0^t G^N\phi\left(\nu^N_s\right)\,\dd s\,\longrightarrow \int_0^t Q_{BGK}^*\varphi(\nu_s)\,\dd s.\end{equation}
\end{enumerate}

\noindent\textbf{Conclusion.} Reporting \eqref{convnu0}, \eqref{convm} and \eqref{convG}  in the representation formula \eqref{representationformula}, we obtain that the weak limit $(\nu_t)_t$ is deterministic and satisfies weakly the space homogeneous BGK equation with initial condition $\nu_0\in\mathcal{P}(\MM)$:
\[\forall \varphi\in C_b(\MM),\,\,\,\,\,\langle \nu_t,\varphi\rangle= \langle \nu_0,\varphi\rangle + \int_0^t Q_{BGK}^*\varphi(\nu_s)\,ds,\]
where the operator $Q_{BGK}^*$ is defined by \eqref{QBGK*}. Since the limit is unique, the whole sequence $(\nu^N_t)_t$ actually converges to $(\nu_t)_t$.
\end{proof}

\begin{corollary}[Convergence of the law of the IBM] Let $(\nu_t)_t$ the weak solution of the space homogeneous equation with initial condition $\nu_0\in\mathcal{P}(\MM)$. Assume that, as $N\to+\infty$,
\[\hat{\mu}^N_0\,\Longrightarrow\,\nu_0.\]
Then, for almost every $t\in\R_+$, it holds that, as $N\to+\infty$,
\[\E_{\hat{\mu}^N_0}\left[\hat{\mu}^N_t\right]\,\Longrightarrow\,\nu_t.\]
\end{corollary}
\begin{proof} Thanks to lemma \ref{IBMmeasureproces}, we know that the law of any particle of the IBM at time $t\in\R_+$ is $\E_{\hat{\mu}^N_0}\left[\hat{\mu}^N_t\right]$ and that this law is equal (in $\mathcal{P}(\MM)$) to $\E_{\hat{\mu}^N_0}\left[\nu^N_t\right]$. The conclusion is therefore a consequence of Theorem \ref{meanfieldGN} and \cite[Theorem III.7.8]{kurtz}.\end{proof}

\section{Conclusion}

We have proved a propagation of chaos property for a Piecewise Deterministic system of agents in a geometrically enriched context. The proof of this property was based on a coupling argument similar to the classical argument \cite{sznitman89} for McKean-Vlasov systems. We have also proved that the under a moderate interaction assumption, as in \cite{jourdainmeleard}, the interaction between the agents can be made purely local, in the sense that the size of the neighbourhood decreases with the number of agents. The resulting kinetic equation is a BGK equation which has been studied in \cite{toto}. Finally, an alternative approach based on martingale arguments can be carried out in the spatially homogenous case. \\ 

This last approach could provide a better bound for the moderate interaction assumption, as in \cite{Oelschl_ger_1985} but it cannot be easily carried out in the spatially inhomogeneous case for the type of interactions studied in the present work. Moreover, our study relies heavily on the compactness of the manifold of orientation and it is not clear how this assumption could be removed. Finally, as in more classical settings, the propagation of chaos property presented here is not uniform in time, see for example \cite{durmus2018elementary,mischler2013kac} for a recent approach to solve this issue. 

\appendix

\section{Proof of lemma \ref{rfkmanifold}}\label{proofRFK}
We are going to prove that there exists a Cauchy sequence in $\mathcal{F}$. \\

By compactness, $\MM$ can be covered by finitely many open sets $\Omega_\alpha$ and we can take a finite atlas $(\Omega_\alpha,\varphi_\alpha,\eta_\alpha)_\alpha$ where $\varphi_\alpha:\Omega_\alpha \to \varphi_\alpha(\Omega_\alpha)\subset\R^p$ and $\eta_\alpha$ is an adapted partition of the unity. For all $\alpha$, on the compact set $\Supp \eta_\alpha\subset\Omega_\alpha$, the metric tensor is bounded in the system of coordinates corresponding to the chart $(\Omega_\alpha,\varphi_\alpha)$ which implies that there exists a constant $C>1$ such that for all $\alpha$ and all $x\in\Supp \eta_\alpha$,
 \begin{equation}\label{Cgij}\frac{1}{C} \leq \sqrt{|g|_{x}}\leq C,\end{equation}
where $|g|$ denotes the determinant of the matrix, the elements of which are the components of $g$ in the chart $(\Omega_\alpha,\varphi_\alpha)$.\\

 For a given $\alpha$, we are going to prove that the set 
 \[\mathcal{F}_\alpha:=\{(x,v)\mapsto \eta_\alpha(\varphi_\alpha^{-1}(v))f(x,\varphi_\alpha^{-1}(v)),\,\,\,f\in\mathcal{F}\}\]
 is relatively compact in $L^1(\R^d\times\varphi_\alpha(\Omega_\alpha))$. It is a consequence of Riesz-Fréchet-Kolmogorov theorem \cite[Corollary 4.27]{brezis2010functional}. Let $\varepsilon>0$ and $R>0$, $\eta>0$ as in the hypothesis.
 \begin{enumerate}[$\bullet$]
 \item The set $\mathcal{F}_\alpha$ is bounded: for all $f\in\mathcal{F}$, it holds that 
 \begin{align*}
 \int_{\R^d\times\varphi_\alpha(\Omega_\alpha)} \eta(\varphi_\alpha^{-1}(v))f(x,\varphi_\alpha^{-1}(v))\dd x\dd v &\leq  C\int_{\R^d\times\varphi_\alpha(\Omega_\alpha)} (\eta_\alpha f(x,\cdot)\sqrt{|g|})\circ\varphi_\alpha^{-1}\dd v\dd x \\
 &\leq C\iint_{\R^d\times\MM} f(x,m)\dd x \dd m.
 \end{align*}
 \item The tightness is given by a similar argument and the first hypothesis. 
 \item Let $\omega\subset\subset\varphi_\alpha(\Omega_\alpha)$ and let $\delta$ be the distance between $\omega$ and the boundary of $\varphi_\alpha(\Omega_\alpha)$. Let $(h,u)\in\R^d\times\R^p$ such that $|h|<\eta$ and $|u|<\min(\eta',\eta'')$ where $\eta'$ and $\eta''$ will be defined later. It holds that: 
 \begin{multline*}
 \iint_{\R^d\times\omega} |\eta_\alpha(\varphi_\alpha^{-1}(v+u))f(x+h,\varphi_\alpha^{-1}(v+u)) -\eta_\alpha(\varphi_\alpha^{-1}(v))f(x,\varphi_\alpha^{-1}(v)|\dd x\dd v  \\
 \leq \iint_{\R^d\times\omega} |\eta_\alpha(\varphi_\alpha^{-1}(v))f(x+h,\varphi_\alpha^{-1}(v+u)) -\eta_\alpha(\varphi_\alpha^{-1}(v))f(x,\varphi_\alpha^{-1}(v)|\dd x\dd v \\ 
 + \iint_{\R^d\times\omega} |\eta_\alpha(\varphi_\alpha^{-1}(v+u))-\eta_\alpha(\varphi_\alpha^{-1}(v))||f|(x+h,\varphi_\alpha^{-1}(v+u))\dd x\dd v.
 \end{multline*}
 For the second term on the right-hand side:
 \[\iint_{\R^d\times\omega} |\eta_\alpha(\varphi_\alpha^{-1}(v+u))-\eta_\alpha(\varphi_\alpha^{-1}(v))||f|(x+h,\varphi_\alpha^{-1}(v+u))\dd x\dd v\leq \varepsilon\]
 for $|u|<\eta'$ where $\eta'$ is smaller than $\delta/2$ and than the $\varepsilon$-modulus of uniform continuity of $\eta_\alpha\circ\varphi_\alpha^{-1}$ on the compact set constructed by enlarging $\omega$ by $\delta/2$. For the first term on the right hand side: 
 \begin{align*}
\hspace{-0.5cm} &\iint_{\R^d\times\omega} |\eta_\alpha(\varphi_\alpha^{-1}(v))f(x+h,\varphi_\alpha^{-1}(v+u)) -\eta_\alpha(\varphi_\alpha^{-1}(v))f(x,\varphi_\alpha^{-1}(v)|\dd x\dd v \\
 &\leq C\iint_{\R^d\times\omega} |\eta_\alpha(\varphi_\alpha^{-1}(v))f(x+h,\varphi_\alpha^{-1}(v+u)) -\eta_\alpha(\varphi_\alpha^{-1}(v))f(x,\varphi_\alpha^{-1}(v)|\\
 &\hspace{9cm}\times\sqrt{|g|}\circ\varphi_\alpha^{-1}(v)\dd x\dd v \\
 &\leq \iint_{\R^d\times\varphi_\alpha^{-1}(\omega)} |f(x+h,\phi(m))-f(x,m)|\dd x\dd m
 \end{align*}
 where for $m\in\varphi_\alpha^{-1}(\omega)$, 
 \[\phi(m):=\varphi_{\alpha}^{-1}(\varphi_\alpha(m)+u)\in\Omega_\alpha\]
 satisfies
 \[d(\phi(m),m)\leq\eta\]
 for $|u|\leq\eta''$ where $\eta''$ is smaller than $\delta/2$ and than the $\eta$-modulus of uniform continuity of $\varphi_\alpha^{-1}$ on the compact set constructed by enlarging $\omega$ by $\delta/2$. Finally, for $|h|<\eta$ and $|u|<\min(\eta',\eta'')$, 
 \[ \iint_{\R^d\times\varphi_\alpha(\Omega_\alpha)} |\eta_\alpha(\varphi_\alpha^{-1}(v+u))f(x+h,\varphi_\alpha^{-1}(v+u)) -\eta_\alpha(\varphi_\alpha^{-1}(v))f(x,\varphi_\alpha^{-1}(v)|\dd x\dd v \leq 2\varepsilon.\]
 \end{enumerate}
Therefore we can apply Riesz-Fréchet-Kolmogorov theorem which shows that $\mathcal{F}_\alpha$ is relatively compact in $L^1(\R^d\times\varphi_\alpha(\Omega_\alpha))$. As a consequence, there exists a Cauchy sequence in $\mathcal{F}_\alpha$.\\

Now consider a sequence $(f_n)_n$ in $\mathcal{F}$. We can extract a subsequence still denoted by $(f_n)_n$ such that for all $\alpha$, the sequence $(\eta_\alpha f_n\circ\varphi_\alpha^{-1})_n$ is a Cauchy sequence in $L^1(\R^d\times \varphi_\alpha(\Omega_\alpha))$. Inequality \eqref{Cgij} implies that for any $\alpha$, the sequence $(\eta_\alpha f_n)_n$ is a Cauchy sequence in the space $L^1(\R^d\times\Omega_\alpha)$. Then, for any $n,m\in\N$, since: 
 \[\|f_n-f_m\|_{L^1(\R^d\times\MM)}\leq\sum_\alpha\|\eta_\alpha f_n-\eta_\alpha f_m\|_{L^1(\R^d\times \Omega_{\alpha})},\]
 we conclude that $(f_n)_n$ is a Cauchy sequence in $L^1(\R^d\times\MM)$ and thus converges since $L^1(\R^d\times\MM)$ is complete. 

\section*{Acknowledgments}
\addcontentsline{toc}{section}{Acknowledgments} The author wishes to thank Sara Merino-Aceituno, Pierre Degond and Amic Frouvelle for drawing his attention to this problem, for their support and useful advice. The author is greatful for the hospitatility of the Faculty of Mathematics at the University of Vienna, where part of this research was conducted and also to the grant Vienna Research Groups for Young Investigators (WWTF) grant number VRG17-014 for funding this visit.

\bibliographystyle{abbrv}
\bibliography{biblio}

\end{document}